\documentclass[12pt]{article}
\usepackage{mathrsfs}
\usepackage{graphicx}
\usepackage{amsmath}
\usepackage{amssymb}
\usepackage{amsthm}

\usepackage{thmtools}
\usepackage{amsfonts}

\usepackage{hyperref}
\hypersetup{
    colorlinks=true,
    linkcolor=blue,
    filecolor=blue,      
    urlcolor=cyan,
}
\usepackage{bbding}
\usepackage{CJK,CJKnumb}
\usepackage{comment}
\usepackage{mathtools}
\usepackage{xcolor}
\usepackage{authblk}

\usepackage{etoolbox}
\patchcmd{\thebibliography}{\chapter*}{\section*}{}{}
\usepackage{ltablex}
\usepackage{enumerate}
\usepackage[nospace,noadjust]{cite}
\usepackage{calc}
\usepackage{geometry}
\usepackage{cite}
\usepackage{tikz-cd}
\usepackage{tikz}
\usetikzlibrary{matrix}

\usepackage{indentfirst}
\usepackage{longtable}

\allowdisplaybreaks

\numberwithin{equation}{section} 

\newtheorem{definition}{Definition}[section]

\newtheorem{theorem}[definition]{Theorem}

\newtheorem{corollary}[definition]{Corollary}
\newtheorem{lemma}[definition]{Lemma}

\newtheorem{example}[definition]{Example}
\newtheorem{remark}[definition]{Remark}

\newcommand{\Qbar}{ \overline{\mathbb{Q}}}
\newcommand{\kbar}{ \overline{K}}

\newcommand{\Q}{\mathbb{Q}}

\newcommand{\bP}{\mathbb{P}}

\newcommand{\bC}{\mathbb{C}}
\newcommand{\bQ}{\mathbb{Q}}

\newcommand{\cO}{\mathcal{O}}

\newcommand{\QQ}{\mathbb{Q}}
\newcommand{\ZZ}{\mathbb{Z}}

\DeclareMathOperator{\vol}{vol}

\DeclareMathOperator{\length}{length}
\DeclareMathOperator{\Supp}{Supp}
\DeclareMathOperator{\dv}{div}

\DeclareMathOperator{\codim}{codim}

\newcommand{\cI}{\mathcal{I}}

\newcommand{\CO}{\mathcal{O}}

\newcommand{\Addresses}{{
  \bigskip
  \footnotesize
  \noindent Keping Huang, \textsc{Institute for Advanced Studies in Mathematics, Harbin Institute of Technology, Haribin, China 150001}\\
    \noindent \textit{E-mail address}: \texttt{kphuang@hit.edu.cn}
\bigbreak
  \noindent Aaron Levin, \textsc{Department of Mathematics, Michigan State University, East Lansing, USA 48824}\\
    \noindent \textit{E-mail address}: \texttt{levina@msu.edu}
\bigbreak
  \noindent Zheng Xiao, \textsc{Beijing International Center for Mathematical Research, Peking University,
    Beijing, China 100871}\\
    \noindent \textit{E-mail address}: \texttt{xiaozheng@bicmr.pku.edu.cn}

}}

\title{A New Diophantine Approximation Inequality on Surfaces and Its Applications}
    
\author
{
KEPING HUANG, AARON LEVIN, AND ZHENG XIAO
}

\date{\vspace{-4em}}

\begin{document}
\maketitle

\def\Z{{\bf Z}}


\vspace{2em}

\begin{abstract} 
\noindent 
We prove a Diophantine approximation inequality for closed subschemes on surfaces which can be viewed as a joint generalization of recent inequalities of Ru-Vojta and Heier-Levin in this context. As applications, we study various Diophantine problems on affine surfaces given as the complement of three numerically parallel ample projective curves: inequalities involving greatest common divisors, degeneracy of integral points, and related Diophantine equations including families of $S$-unit equations.  We state analogous results in the complex analytic setting, where our main result is an inequality of Second Main Theorem type for surfaces, with applications to the study and value distribution theory of holomorphic curves in surfaces.
\end{abstract}

\section{Introduction}

Schmidt's Subspace Theorem occupies a central place in the theory of Diophantine approximation, and it is known to have deep and far-reaching applications. Stated in the language of heights (and with improvements due to Schlickewei \cite{Schl77} allowing for arbitrary finite sets of places), the Subspace Theorem may be stated as follows:

\begin{theorem}[Schmidt Subspace Theorem]
\label{subspace}
Let $S$ be a finite set of places of a number field $K$.  For each $v\in S$, let $H_{0,v},\ldots,H_{n,v}\subset \mathbb{P}^n$ be hyperplanes over $K$ in general position.  Let $\varepsilon>0$.  Then there exists a finite union of hyperplanes $Z\subset \mathbb{P}^n$ such that for all points $P\in \mathbb{P}^n(K)\setminus Z$,
\begin{equation*}
\sum_{v\in S}\sum_{i=0}^n \lambda_{H_{i,v},v}(P)< (n+1+\varepsilon)h(P).
\end{equation*}
\end{theorem}

Here, $\lambda_{H_{i,v},v}$ is a local height function (also known as a local Weil function) associated to the hyperplane $H_{i,v}$ and place $v$ in $S$, and $h$ is the standard (logarithmic) height on $\mathbb{P}^n$.


The Subspace Theorem has been generalized to the setting of hypersurfaces in projective space by Corvaja and Zannier \cite{CZ04'}, and more generally, by Corvaja and Zannier \cite{CZ04'} (for complete intersections) and Evertse and Ferretti \cite{EF08} (for arbitrary projective varieties) to divisors which possess a common linearly equivalent multiple. Building on work of Autissier \cite{Aut11}, Ru and Vojta \cite{RV20} proved a general version of the Subspace Theorem in terms of beta constants (see Definition \ref{betadef}). Their inequality was subsequently extended to the context of closed subschemes by Ru and Wang \cite{RW22} and by Vojta \cite{Voj23} (under slightly different intersection conditions). We state the following general form of the inequality, due to Vojta \cite{Voj23} (in fact, Vojta proves a  stronger version of Theorem \ref{theoremRV} with the condition ``intersect properly" (Definition \ref{properdef}) replaced by ``weakly intersect properly" \cite[Definition ~4.1(c)]{Voj23}).

\begin{theorem}[Ru-Vojta \cite{RV20}, Ru-Wang \cite{RW22}, Vojta  \cite{Voj23}]
\label{theoremRV}
Let $X$ be a projective variety of dimension $n$ defined over a number field $K$.  Let $S$ be a finite set of places of $K$.  For each $v\in S$, let $Y_{0,v},\ldots, Y_{n,v}$ be closed subschemes of $X$, defined over $K$, that intersect properly.  Let $A$ be a big divisor on $X$, and let $\varepsilon>0$.  Then there exists a proper Zariski-closed subset $Z\subset X$ such that for all points $P\in X(K)\setminus Z$,
\begin{equation*}
\sum_{v\in S}\sum_{i=0}^{n} \beta(A, Y_{i,v}) \lambda_{Y_{i,v},v}(P)< (1+\varepsilon)h_A(P).
\end{equation*}
\end{theorem}

Heier and the second author \cite{HL21} proved an inequality using Seshadri constants (Definition \ref{Seshadridef}) in place of beta constants, and with the proper intersection condition of Theorem \ref{theoremRV} replaced by a flexible notion of general position for closed subschemes (Definition \ref{generalpositiondef}).

\begin{theorem}[Heier-Levin \cite{HL21}]
\label{theoremHL}
Let $X$ be a projective variety of dimension $n$ defined over a number field $K$.  Let $S$ be a finite set of places of $K$.  For each $v\in S$, let $Y_{0,v},\ldots, Y_{n,v}$ be closed subschemes of $X$, defined over $K$, and in general position.  Let $A$ be an ample Cartier divisor on $X$, and $\varepsilon>0$.  Then there exists a proper Zariski-closed subset $Z\subset X$ such that for all points $P\in X(K)\setminus Z$,
\begin{equation*}
\sum_{v\in S}\sum_{i=0}^n \epsilon(A,Y_{i,v}) \lambda_{Y_{i,v},v}(P)< (n+1+\varepsilon)h_A(P).
\end{equation*}
\end{theorem}

In both Theorem \ref{theoremRV} and Theorem \ref{theoremHL}, one may also apply the results with fewer than $n+1$ chosen closed subschemes at $v$ (e.g., by appropriately arbitrarily choosing the remaining closed subschemes and using positivity of the associated local heights outside a closed subset); we will use this fact without further remark. 

Before stating our main result, we give some further remarks comparing Theorem~\ref{theoremRV} and Theorem~\ref{theoremHL}. For simplicity and to avoid technical issues, we assume now for the discussion that $X$ is nonsingular of dimension $n$ and we fix an ample divisor $A$ on $X$. If the closed subschemes $Y_{i,v}=D_{i,v}$ are all ample effective divisors, then by our nonsingularity assumption, $D_{0,v},\ldots, D_{n,v}$ intersect properly if and only if they are in general position (Remark \ref{remgenproper}). Furthermore, we note that if $D_{i,v}\sim d_{i,v} A$, then it follows easily from the asymptotic Riemann-Roch formula and the definitions that
\begin{align*}
\beta(A,D_{i,v})&=\frac{1}{(n+1)d_{i,v}},\\
\epsilon(A,D_{i,v})&=\frac{1}{d_{i,v}}.
\end{align*}
In this case, Theorem \ref{theoremRV} and Theorem \ref{theoremHL} coincide and yield the inequality (outside a Zariski-closed subset)
\begin{align*}
\sum_{v\in S}\sum_{i=0}^n \frac{\lambda_{D_{i,v},v}(P)}{d_{i,v}}< (n+1+\varepsilon)h_A(P),
\end{align*}
which is precisely the inequality of the aforementioned theorem of Evertse and Ferretti.

Another interesting example comes from considering a single closed subscheme $Y$ of $X$, which we assume additionally to be a local complete intersection (to satisfy the proper intersection hypothesis of Theorem \ref{theoremRV}). In this case, if $\codim Y=r$, from \cite{HL21} we have the inequality
\begin{align*}
\beta(A,Y)\geq \frac{r}{n+1}\epsilon(A,Y),
\end{align*}
and so Theorem \ref{theoremRV} yields (choosing $Y_{0,v}=Y$ for all $v$)
\begin{align}
\label{ineqMR}
\sum_{v\in S}\epsilon(A,Y)\lambda_{Y,v}(P)< \left(\frac{n+1}{r}+\varepsilon\right)h_A(P).
\end{align}
This is a generalization of an inequality of McKinnon and Roth \cite{MR15} (who proved the case when $Y$ is a point). On the other hand, in Theorem \ref{theoremHL} one may take $Y_{0,v}=\cdots =Y_{r-1,v}=Y$ for all $v$ (i.e., repeat $Y$ as the choice of closed subscheme $r$ times), as is allowed by the used notion of general position (Definition \ref{generalpositiondef}), and one again obtains the McKinnon-Roth type inequality \eqref{ineqMR}. In short, there seems to be a tradeoff between the flexible notion of general position in Theorem \ref{theoremHL} and the ``better" beta constant of Theorem \ref{theoremRV} (as compared to the Seshadri constant). 


The goal of our main result is to provide, in the case of surfaces, an inequality which combines the separate advantages of Theorem \ref{theoremRV} and Theorem \ref{theoremHL}; we prove an inequality in terms of beta constants of closed subschemes, under hypotheses which permit the use of nested closed subschemes (as allowed in Theorem \ref{theoremHL}, but not in Theorem \ref{theoremRV}).

\begin{theorem}[Main Theorem]\label{BetaMain}
Let $X$ be a projective surface defined over a number field $K$. Let $S$ be a finite set of places of $K$. For each $v\in S$, let $D_v\supset Y_v$ be a regular chain of nonempty closed subschemes of $X$ (see Definition \ref{regulardef}). Let $A$ be a big Cartier divisor on $X$, and let $\varepsilon > 0$. Then there exists a proper Zariski-closed subset $Z\subset X$ such that for all $P\in X(K)\setminus Z$, we have 
\begin{equation}
\label{BetaMainineq}
    \begin{aligned}
        \sum_{v\in S} \bigg(\beta(A,D_v) 
        \lambda_{D_v,v}(P) + 
       \left( \beta(A, Y_v) - \beta(A, D_v) \right) \lambda_{Y_v,v}(P) \bigg) < (1 + \varepsilon) h_A(P).
    \end{aligned}
\end{equation}
\end{theorem}

In Section \ref{sec: main}, we show that our main result implies the Ru-Vojta inequality for surfaces, and Theorem \ref{theoremHL} for surfaces under regularity assumptions.  Thus, we may view Theorem \ref{BetaMain} as a kind of joint generalization, for surfaces, of the inequalities of Ru-Vojta and of Heier and the second author.

We remark that the regular chain assumption is quite natural from the point of view of classical Diophantine approximation, where one is interested in sums of local heights associated to divisors. Indeed, if $D_1,\ldots, D_q$ are effective divisors in general position on a nonsingular projective surface $X$, all defined over a number field $K$, $S$ is a set of places of $K$, and we fix a point $P\in X(K)$, then from general position a point can be $v$-adically close to at most two divisors $D_i$. Thus,
\begin{align*}
\sum_{v\in S}\sum_{i=1}^q \lambda_{D_i,v}(P)=\sum_{v\in S} \left(\lambda_{D_{i_v},v}(P)+\lambda_{D_{j_v},v}(P)\right) +O(1),
\end{align*}
for some choice of $i_v,j_v\in \{1,\ldots, q\}$ depending on $P$ and $v$ (but with the constant in the $O(1)$ independent of $P$). Now assuming $\lambda_{D_{i_v},v}(P)\geq \lambda_{D_{j_v},v}(P)$ (as we may), from a basic property of heights associated to closed subschemes (see Subsection \ref{subsec: subscheme}) we have (up to $O(1)$)
\begin{align*}
\lambda_{D_{i_v}\cap D_{j_v},v}(P)=\min\{\lambda_{D_{i_v},v}(P),\lambda_{D_{j_v},v}(P)\}=\lambda_{D_{j_v},v}(P).
\end{align*}
Therefore, we may write
\begin{align}
\label{regseqarg}
\sum_{v\in S}\sum_{i=1}^q\lambda_{D_i,v}(P)=\sum_{v\in S}
\left(\lambda_{D_{i_v},v}(P)+\lambda_{D_{i_v}\cap D_{j_v},v}(P)\right)+O(1),
\end{align}
where now $D_{i_v}\supset D_{i_v}\cap D_{j_v}$ is a regular chain of closed subschemes of $X$ for all $v\in S$. Although the inequality \eqref{regseqarg} is essentially elementary (relying at its base on Hilbert's Nullstellensatz), we will find the underlying argument useful in a number of proofs and applications.

Among the potential applications of our main inequality \eqref{BetaMainineq}, in Section \ref{sec:applications} we study inequalities related to greatest common divisors, the Diophantine equation $f(a^m,y)=b^n$, integral points on certain affine surfaces, and solutions to families of unit equations. For instance, we prove the following result for integral points on certain complements of three curves in $\mathbb{P}^2$:

\begin{theorem}\label{nonGeneralP2}
Let $D_1, D_2, D_3$ be distinct irreducible projective curves in $\mathbb{P}^2$, defined over a number field $K$, of degrees $d_1,d_2,d_3$, respectively, such that
\begin{align*}
D_1\cap D_2\cap D_3\neq \emptyset,
\end{align*}
and for every point $Q\in (D_1\cap D_2\cap D_3)(\kbar)$ and $1\leq i<j\leq 3$, 
\begin{align}
\label{49ineq}
(D_i.D_j)_Q<\frac{4}{9}d_id_j,
\end{align}
where $(D_i.D_j)_Q$ denotes the local intersection multiplicity of $D_i$ and $D_j$ at $Q$.  Let $S$ be a finite set of places of $K$ containing all the archimedean places. Then there exists a proper Zariski-closed subset $Z\subset \mathbb{P}^2$ such that for any set $R\subset X(K)$ of $(D_1 + D_2 + D_3, S)$-integral points, the set $R\setminus Z$ is finite.
\end{theorem}

When $d_1=d_2=d_3>1$ and $D_1\cap D_2\cap D_3$ consists of a single point, at which the divisors intersect transversally, the result follows from work of Corvaja and Zannier \cite{CZ06} (see Theorem \ref{CZnongen}). The factor $\frac{4}{9}$ in the intersection condition \eqref{49ineq} of Theorem~\ref{nonGeneralP2} cannot be replaced by anything larger than $\frac{1}{2}$ (Example \ref{12example}), and in fact a more general version of the theorem stated in terms of beta constants (Theorem \ref{nonGeneral}) can be shown to be sharp (Example \ref{beta1example}). We note that in general it seems to be a difficult problem to prove the degeneracy of integral points on the complement of three curves in $\mathbb{P}^2$. For instance, if $D=D_1+D_2+D_3$ is a normal crossings divisor of degree $\deg D\geq 4$ with $D_1, D_2, D_3$ geometrically integral curves in $\mathbb{P}^2$ (all defined over some number field $K$), then a suitable version of Vojta's conjecture predicts the same conclusion as in Theorem \ref{nonGeneralP2}, but this does not appear to be known (for all $K$ and $S$) for even a single such divisor $D$.

For a further application, by applying a general form of the above theorem (Theorem \ref{nonGeneral}) to certain surfaces, we study families of unit equations of the form:
\begin{align*}
f_1(t)u+f_2(t)v=f_3(t), \quad t\in \CO_{K,S}, u,v\in \CO_{K,S}^*,
\end{align*}
where $f_1,f_2,f_3$ are polynomials in $K[t]$, and $\CO_{K,S}^*$ is the group of $S$-units of a ring of $S$-integers $\CO_{K,S}$ of a number field $K$. This equation was treated in the case $\deg f_1=\deg f_2=\deg f_3$ by Corvaja and Zannier \cite{CZ06, CZ10}, and in the case $\deg f_1+\deg f_2=\deg f_3$ by the second author \cite{Lev06}. We prove a result whenever the degrees of the polynomials are roughly within a factor of $2$ of each other.

\begin{theorem}
Let $f_1,f_2,f_3\in K[t]$ be nonconstant polynomials without a common zero of degrees $d_1,d_2,d_3$, respectively, and suppose that
\begin{align*}
\max_i (d_i+1)<\frac{9}{5}\min_i d_i.
\end{align*}
Then the set of solutions $(t,u,v)\in\mathbb{A}^3(K)$ of the equation
\begin{align*}
f_1(t)u+f_2(t)v=f_3(t), \quad t\in \CO_{K,S}, u,v\in \CO_{K,S}^*,
\end{align*}
is contained in a finite number of rational curves in $\mathbb{A}^3$.
\end{theorem}

Excluding the linear case (proved in \cite{CZ06}), this recovers Corvaja-Zannier's result \cite{CZ10} (i.e., the case $d_1=d_2=d_3$).

Finally, we mention that via Vojta's dictionary \cite[Ch.~3]{Voj87} between Diophantine approximation and Nevanlinna theory, by substituting Vojta's version \cite{Voj97} of Cartan's Second Main Theorem in place of Schmidt's Subspace Theorem, one can prove a result analogous to Theorem \ref{BetaMain}, giving the following inequality in the vein of the Second Main Theorem:
\begin{theorem}
\label{NevBetaMain}
Let $X$ be a complex projective surface. Let $q$ be a positive integer and let $D_i\supset Y_i$ be a regular chain of nonempty closed subschemes of $X$ for $i=1,\ldots, q$. Let $A$ be a big Cartier divisor on $X$, and let $\varepsilon > 0$.  Then there exists a proper Zariski-closed subset $Z\subset X$ such that for all holomorphic maps $f:\bC \to X$ whose image is not contained in $Z$, the inequality 
\begin{align*}
\int_{0}^{2\pi}\max_i \bigg(\beta(A,D_i)\lambda_{D_i}(f(re^{i\theta})) + 
       ( \beta(A, Y_i) - \beta(A, D_i) )   \lambda_{Y_i}(f(re^{i\theta}))\bigg)&  \frac{d\theta}{2\pi}\\  
       &< (1 + \varepsilon) T_{f,A}(r).
\end{align*}
holds for all $r \in (0, \infty)$ outside of a set of finite Lebesgue measure.
\end{theorem}

One may similarly apply Theorem \ref{NevBetaMain} to obtain results on holomorphic curves. For instance, using Vojta's dictionary, one may use Theorem \ref{NevBetaMain} to prove the following analogue of Theorem \ref{nonGeneralP2} for holomorphic curves:

\begin{theorem}
Let $D_1, D_2, D_3$ be distinct irreducible complex projective curves in $\mathbb{P}^2$, of degrees $d_1,d_2,d_3$, respectively, such that
\begin{align*}
D_1\cap D_2\cap D_3\neq \emptyset,
\end{align*}
and for every point $Q\in (D_1\cap D_2\cap D_3)(\mathbb{C})$ and $1\leq i<j\leq 3$, 
\begin{align*}
(D_i.D_j)_Q<\frac{4}{9}d_id_j,
\end{align*}
where $(D_i.D_j)_Q$ denotes the local intersection multiplicity of $D_i$ and $D_j$ at $Q$.  Then there exists a proper Zariski-closed subset $Z\subset \mathbb{P}^2$ such that every nonconstant holomorphic map $f:\mathbb{C}\to X\setminus (D_1\cup D_2\cup D_3)$ has image contained in $Z$.
\end{theorem}

As the proofs of these analogous results in the complex setting are similar to the proofs of their arithmetic counterparts (after making the appropriate ``translations"), we omit the details.

The organization of this paper is as follows. Section \ref{sec: back} gives relevant background material, including a summary of Silverman's theory of heights associated to closed subschemes, and some needed results in commutative algebra, algebraic geometry, and Diophantine approximation. In Section \ref{sectionbeta} we give some basic inequalities involving beta constants, and then in Section \ref{sec: main} we prove our main theorem and show how it implies, for surfaces, versions of the inequality of Ru-Vojta and the inequality of Heier and the second author. Finally, in Section \ref{sec:applications} we give some illustrative Diophantine applications of our results.


\section{Background}\label{sec: back}

\subsection{Heights Associated to Closed Subschemes	}\label{subsec: subscheme}

In \cite{Sil87}, Silverman generalized the Weil height machine for Cartier divisors to height functions on projective varieties with respect to closed subschemes. 
More precisely, let $X$ be a projective variety over a number field $K$, and let $Z(X)$ denote the set of closed subschemes of $X$. Let $M_K$ be the set of places of $K$.
Note that the closed subschemes $Y \in Z(X)$ are in one-to-one correspondence with quasi-coherent ideal sheaves 
$\mathcal{I}_Y \subseteq \mathcal{O}_X$, and we identify a closed subscheme $Y$ with its ideal
sheaf $\mathcal{I}_{Y}$. Generalizing the Weil height machine for Cartier divisors, Silverman assigned to each $Y \in Z(X)$ and each place $v\in M_K$ a local height function $\lambda_{Y,v}$, 
and to each $Y \in Z(X)$ a global height function $h_Y=\sum_{v\in M_K}\lambda_{Y,v}$ (both uniquely determined up to a bounded function). We now summarize some of the basic properties of height functions associated to closed subschemes.

\begin{theorem}\label{thm: Silverman}(\cite{Sil87})
Let $X$ be a projective variety over a number field $K$. Let $Z(X)$ be the set of closed subschemes of $X$. There are maps
\begin{equation*}
    \begin{aligned}
Z(X) \times M_K &\rightarrow \{\text{functions }X(K) \rightarrow [0,+\infty] \},\\
(Y,v)&\mapsto \lambda_{Y,v},\\
Z(X) &\rightarrow \{\text{functions }X(K) \rightarrow [0,+\infty] \},\\
Y&\mapsto h_Y,
    \end{aligned}
\end{equation*}
satisfying the following properties (we also write $\lambda_{X,Y,v}$ and $h_{X,Y}$ for clarity in \eqref{subfun}):
\begin{enumerate}
		\item If $D\in Z(X)$ is an effective Cartier divisor, then $\lambda_{D,v}$ and $h_D$ agree with the classical height functions associated to $D$.
    \item   If $W,Y \in  Z(X)$ satisfy $W \subseteq Y$, then $h_W \le h_Y+O(1)$ and $\lambda_{W,v} \le \lambda_{Y,v}+O(1)$ for all $v\in M_K$. 
    \item  If $ W,Y \in Z(X)$ satisfy $\mathrm{Supp}(W) \subseteq \mathrm{Supp}(Y)$,
    then there exists a constant $C$  
such that $h_W \le C \cdot h_Y+O(1)$ and $\lambda_{W,v} \le C \cdot \lambda_{Y,v}+O(1)$ for all $v\in M_K$. 
\item For all $W, Y \in Z(X)$,
$\lambda_{W \cap Y, v}  = \min \{\lambda_{W,v}, \lambda_{Y,v}\}+O(1).$\label{subint}
\item For all $W,Y\in Z(X)$, we have $h_{W+Y}=h_W+h_Y+O(1)$ and $\lambda_{W+Y,v}=\lambda_{W,v}+\lambda_{Y,v}+O(1)$ for all $v\in M_K$.
\item Let $\phi: X' \rightarrow X$ be a morphism of projective varieties over $K$, and let $Y \in Z(X)$. Then
		\begin{align*}
		h_{X',\phi^*Y} &= h_{X, Y}\circ \phi+O(1),\\
		\lambda_{X', \phi^*Y,v} &= \lambda_{X, Y,v}\circ \phi+O(1),
		\end{align*}
		for all $v\in M_K$. \label{subfun}
\item If $D$ and $E$ are numerically equivalent Cartier divisors on $X$ and $A$ is an ample divisor on $X$, then for any $\varepsilon>0$, we have
\begin{align*}
|h_D(P)-h_E(P)|<\varepsilon h_A(P)+O(1)
\end{align*}
for all $P\in X(K)$. \cite[Proposition 1.2.9(d)]{Voj87}.
\end{enumerate}
\end{theorem}

Here, $Y\subset Z$, $Y+Z$, and $\phi^*Y$ are all defined in terms of the associated ideal sheaves (see \cite{Sil87}). For a closed subscheme $Y$ and finite set of places of $S$ of $K$, we let $m_{Y,S}(P)=\sum_{v\in S}\lambda_{Y,v}(P)$. For Cartier divisors $D$ and $E$ on a variety $X$, we will also write $D\geq E$ (or $E\leq D$) if $D-E$ is an effective divisor.

\begin{remark}
\label{remzerodim}
If $Y$ is a zero-dimensional closed subscheme of a projective variety $X$ over a field $K$, with $\Supp Y=\{Q_1,\ldots, Q_r\}$, then we can write
\begin{align*}
Y=Y_{Q_1}+\cdots +Y_{Q_r},
\end{align*}
where $Y_{Q_i}$ is a closed subscheme supported only at $Q_i$ (see, e.g., \cite[Prop.~5.11]{UW}).  Suppose that $K$ is a number field and fix $v\in M_K$. Since $Y_{Q_i}\cap Y_{Q_j}=\emptyset$, $i\neq j$, it follows from Theorem \ref{thm: Silverman} \eqref{subint} that 
\begin{align*}
\min \{\lambda_{Y_{Q_i},v}, \lambda_{Y_{Q_j},v}\}=O(1). 
\end{align*}
Since we also have $\lambda_{Y,v}=\sum_{i=1}^r \lambda_{Y_{Q_i},v}+O(1)$, for any point $P\in X(K)$ there exists $j\in \{1,\ldots,r\}$ (depending on $P$ and $v$) such that
\begin{align*}
\lambda_{Y,v}(P)=\lambda_{Y_{Q_j},v}(P)+O(1),
\end{align*}
where the $O(1)$ is independent of $P$.
\end{remark}

\begin{definition}\label{defzerodim}
If $Y$ is a zero-dimensional closed subscheme of a projective variety $X$ over $K$ and $Q\in Y(\kbar)$, we let $Y_Q$ be the closed subscheme $(Y_{\kbar})_Q$ of Remark \ref{remzerodim} (applied to the closed subscheme $Y_{\kbar}$ of $X_{\kbar}$).
\end{definition}
\noindent When $Q\in Y(K)$, we also naturally identify $Y_Q$ as a closed subscheme over $K$.

We also recall the notion of a set of $(D,S)$-integral points on a projective variety $X$ (see \cite{Voj87} for more details).

\begin{definition}\label{def: integral}
Let $D$ be an effective Cartier divisor on a projective variety $X$, both defined over a number field $K$. Let $R$ be a subset of 
$X(K) \setminus \Supp D$, and let $S$ be a finite set of places of $K$ containing all the archimedean places. 
We say that $R$ is a set of $(D, S)$-integral points on $X$ if
\begin{align*}
\sum_{v\in S}\lambda_{D,v}(P)=h_D(P)+O(1)
\end{align*}
for all $P\in R$.
\end{definition}

More generally, this definition can be extended to an arbitrary closed subscheme $Y$ of $X$ by appropriately replacing $D$ with $Y$ everywhere.

\subsection{Multiplicities, Regular Sequences, and the Filtration Lemma}

In this section, we collect together some useful facts involving multiplicities and regular sequences, and state the Filtration Lemma which will play a basic role in many of the proofs.

Let $R$ be a Noetherian local ring and let $\mathfrak{q}$ be a parameter ideal for $R$. Let $L_{\mathfrak{q},R}(n)=\length (R/\mathfrak{q}^n)$.
\begin{theorem}[{\cite[Section 12.4]{Eis95}}]
For $n\gg 0$, $L_{\mathfrak{q},R}(n)$ agrees with a polynomial of degree $\dim R$ in $n$.
\end{theorem}

The leading term of any polynomial of degree $d$ in $\mathbb{Q}[n]$ which takes only integer values for $n\gg 0$ is of the form  $cn^d/d!$ for some integer $c$. Then we have the following notions of multiplicity going back to Hilbert and Samuel.

\begin{definition}[\cite{Eis95}]
We define the \emph{multiplicity} of $\mathfrak{q}$ in $R$ to be the integer $e=e(\mathfrak{q},R)$ such that $|L_{\mathfrak{q},R}(n)-en^{\dim R}/(\dim R)!|=O(n^{\dim R-1}).$
\end{definition}

\begin{definition}[{\cite[Ex.~4.3.4]{Fulton}}]
Let $Z$ be a zero-dimensional closed subscheme of a projective variety $X$ with ideal sheaf $\mathcal{I}$ and let $P\in \Supp Z$. We define the \emph{multiplicity} of $Z$ at $P$ to be $e(\mathcal{I}_P, \cO_{X,P})$, and denote it by $\mu_P(Z)$ or simply $\mu_P$ if the context is clear.
\end{definition}

We also need some facts about certain ideal quotients.

\begin{definition}
Let $I$ and $J$ be ideals in $R$. The ideal quotient (or colon ideal) $(I:J)$  is defined by
\begin{align*}
(I:J)=\{r\in R\mid rJ\subset I\}.
\end{align*}
\end{definition}

\begin{lemma}
\label{ideallemma}
Let $I=(f_1,\ldots, f_m)$ be an ideal of a ring $R$ generated by a regular sequence $f_1,\ldots, f_m$ in $R$. Then for any positive integers $n\geq i>0$ and $j\in \{1,\ldots, m\}$, we have
\begin{align*}
(I^n:(f_j)^i)=I^{n-i}.
\end{align*}
\end{lemma}

\begin{proof}
Since $f_j^i\in I^i$, $I^{n-i}\subset (I^n:(f_j)^i)$ is clear. 

For the other inclusion, suppose that $af_j^i\in I^n$. Consider the associated graded ring $G_I(R) := \bigoplus_{r \ge 0} I^r / I^{r+1}$. By \cite[Theorem 27]{Matsumura}, $G_I(R)$ is isomorphic to the polynomial ring $(R/I)[x_1,\ldots, x_m]$, where the indeterminates $x_1,\ldots, x_m$ correspond to the images of $f_1,\ldots, f_m$ in the first graded piece of $G_I(R)$.

Suppose $a \in I^\ell \setminus I^{\ell+1}$, so that $0\neq \bar{a} \in I^\ell / I^{\ell+1}$. Since $x_j$ is not a zero divisor in $(R/I)[x_1,\ldots, x_m]$, we have $0 \neq \overline{a f_j^i} \in I^{\ell + i} / I^{\ell + i + 1}$.  Therefore $af_j^i \notin  I^{\ell + i + 1}$ and it follows that $\ell + i + 1 \ge n+1$. 
Hence $\ell \ge n - i$. 
\end{proof}

We now make definitions involving proper intersections, general position, and regular sequences, in the context of closed subschemes. 

\begin{definition}
\label{properdef}
Let $D_1,\ldots, D_r$ be effective Cartier divisors on a projective variety $X$. We say that $D_1,\ldots, D_r$ intersect properly if for every nonempty subset $I\subset \{1,\ldots, r\}$ and every point $P\in \Supp \cap_{i\in I}D_i$, the sequence $(f_i)_{i\in I}$ is a regular sequence in the local ring $\cO_{X,P}$, where $f_i$ locally defines $D_i$ at $P$, $i=1,\ldots, r$. More generally, the definition can be extended to closed subschemes $Y_1,\ldots, Y_r$ of $X$ (see \cite{RW22} or \cite{Voj23}).
\end{definition}

\begin{remark}
If $D_1$ and $D_2$ are Cartier divisors on $X$ that intersect properly, then $\cO_X(-D_1-D_2)= \cO_X(-D_1)\cap \cO_X(-D_2)$ (see \cite[\S 3, \S 4]{Voj23} for a vast abstraction of this property).
\end{remark}

\begin{definition}
\label{generalpositiondef}
Let $X$ be a projective variety of dimension $n$. We say that closed subschemes $Y_1,\ldots, Y_q$ of $X$ are in general position if for every subset $I \subset \{1,\ldots, q\}$ with $|I| \leq n+1$ we have $\emph{codim} \cap_{i \in I}Y_i \geq |I|$, where we use the convention that $\dim \emptyset = -1$.
\end{definition}

\begin{remark}
\label{remgenproper}
If the Cartier divisors $D_1,\ldots, D_r$ on $X$ intersect properly, then they are in general position. If $X$ is Cohen-Macaulay then the converse holds \cite[Theorem 17.4]{Matsumura}. In particular, if $X\subset \mathbb{P}^n$ is a hypersurface, $D_1,\ldots, D_r$ are effective divisors on $\mathbb{P}^n$, and $X, D_1,\ldots, D_r$ are in general position (on $\mathbb{P}^n$), then $X, D_1,\ldots, D_r$ intersect properly, and from the definitions it follows that $D_1|_X,\ldots, D_r|_X$ intersect properly, as divisors on $X$.
\end{remark}

We make a related definition for a sequence of nested closed subschemes.

\begin{definition}
\label{regulardef}
Let $Y_1\supset Y_2\supset \cdots\supset Y_r$ be closed subschemes of a projective variety $X$. We say that $Y_1\supset\cdots\supset Y_r$ is a regular chain of closed subschemes of $X$ if for every $i\in \{1,\ldots, r\}$ and every point $P\in \Supp Y_i$, there is a regular sequence $f_1,\ldots, f_i\in \cO_{X,P}$ such that for $1\leq j\leq i$, the ideal sheaf of $Y_j$ is locally given by the ideal $(f_1,\ldots, f_j)$ in $\cO_{X,P}$.
\end{definition}

\begin{remark}
If $D_1,\ldots, D_r$ intersect properly on $X$, then $D_1\supset D_1\cap D_2\supset\cdots\supset D_1\cap D_2\cap\cdots \cap D_r$ is a regular chain of closed subschemes of $X$.
\end{remark}

We will use the following consequence of Lemma \ref{ideallemma}.

\begin{definition}
For an effective Cartier divisor $D$ and closed subscheme $Y$ of $X$, let $n_Y(D)$ be the largest nonnegative integer such that $n_Y(D)Y\subset D$, as closed subschemes.
\end{definition}

\begin{lemma}
\label{orderlemma}
Let $D,D'$ be effective Cartier divisors on a surface $X$ and let $Y$ be a closed subscheme supported at a point $P$.  Suppose that $mD\supset Y$ is a regular chain of closed subschemes for some positive integer $m$.  Then for any positive integer $i$,
\begin{align}
\label{orderlemmaineq}
|n_Y(iD)+n_Y(D')-n_Y(iD+D')|\leq 1,
\end{align}
and 
\begin{align*}
n_Y(iD)=\lfloor i/m\rfloor.
\end{align*}

\end{lemma}

\begin{proof}
Let $\mathcal{I}$ be the ideal sheaf of $Y$ and let $\mathcal{I}_P\subset \cO_{P}$ be the localization. Let $n=n_Y(iD+D')$. Suppose that $D$ and $D'$ are represented by $f$ and $f'$ locally at $P$. Then since $nY\subset iD+D'$, we have $f^if'\in \mathcal{I}_P^n$. It follows that $(f^m)^{\lceil i/m\rceil}f'\in \mathcal{I}_P^n$ and, since $mD\supset Y$ is a regular chain of closed subschemes, by Lemma \ref{ideallemma}, $f'\in \mathcal{I}_P^{n-\lceil i/m\rceil}$. On the other hand, since $f^i=(f^{m})^{\lfloor i/m\rfloor}f^{i-m\lfloor i/m\rfloor}\in \mathcal{I}_P^{\lfloor i/m\rfloor}$, if $f'\in \mathcal{I}_P^{n-\lfloor i/m\rfloor+1}$ then  $f^if'\in \mathcal{I}_P^{n+1}$, contradicting the definition of $n$. It follows that 
\begin{align*}
n-\lceil i/m\rceil\leq n_Y(D')\leq n-\lfloor i/m\rfloor. 
\end{align*}
The inequality \eqref{orderlemmaineq} (and the remainder of the lemma) now follows if we show that $n_Y(iD)=\lfloor i/m\rfloor$. If $i$ divides $m$ and $iD=\frac{i}{m}(mD)$ then it follows from Lemma \ref{ideallemma} that $n_Y(iD)=i/m$.  Since $\lfloor i/m\rfloor (mD)\subset iD\subset (\lfloor i/m\rfloor+1) (mD)$, we have $\lfloor i/m\rfloor\leq n_Y(iD)\leq \lfloor i/m\rfloor+1$. But if $n_Y(iD)=\lfloor i/m\rfloor+1$, then 
\begin{align*}
n_Y(miD)\geq mn_Y(iD)>i,
\end{align*}
contradicting $n_Y(miD)=i$. Therefore we must have $n_Y(iD)= \lfloor i/m\rfloor$ as claimed.
\end{proof}

Finally, we recall the elementary, but very useful, Filtration Lemma whose utility in Diophantine approximation was recognized and introduced by Corvaja and Zannier \cite[Lemma~3.2]{CZ04}.

\begin{lemma}[Filtration Lemma]\label{filtrationlemma}
Let $V$ be a vector space of finite dimension $d$ over a field $K$. Let $V=W_1 \supset W_2 \supset \cdots \supset W_h$ and $V=W_1^* \supset W_2^* \supset \cdots \supset W_{h^*}^*$ be two filtrations on $V$. There exists a basis $v_1, \ldots, v_d$ of $V$ that contains a basis of each $W_j$ and $W_{j}^*$.
\end{lemma}

In the situation of Lemma \ref{filtrationlemma}, we say that the basis $v_1,\ldots, v_d$ of $V$ is adapted to the two filtrations.

\section{Beta constant estimates}
\label{sectionbeta}

Throughout this section we work over a field $K$ of characteristic $0$. For a line bundle $\mathscr{L}$ on a projective variety $X$ over $K$, we will write $h^0(X,\mathscr{L})$ (or simply $h^0(\mathscr{L})$) for $\dim_K H^0(X,\mathscr{L})$, and if $D$ is a Cartier divisor on $X$, we write $h^0(D)$ for $h^0(\CO(D))$. We first recall the definition of the beta constant.
\begin{definition}
\label{betadef}
Let $X$ be a projective variety over a field $K$. Let $\mathscr{L}$ be a big line bundle on $X$ and let $Y$ be a closed subscheme of $X$ with associated sheaf of ideals $\mathcal{I}$. Then
\begin{align*}
    \beta(\mathscr{L},Y)=\liminf_{N \to \infty}\frac{\sum_{m=1}^{\infty}h^0(X,\mathscr{L}^N\otimes \mathcal{I}^m)}{Nh^0(X,\mathscr{L}^N)},
\end{align*}
where $\mathscr{L}^N$ denotes the tensor product of $N$ copies of $\mathscr{L}$.
\end{definition}
By a result of Vojta \cite{Voj20} (assuming as we do that $K$ has characteristic $0$) the $\liminf$ in the definition can be replaced by a limit. By abuse of notation, if $D$ is a big Cartier divisor on $X$, we also write $\beta(D,Y)$ for $\beta(\mathcal{O}(D),Y)$.

It follows easily that for any positive integer $n$,
\begin{align*}
\beta(nD,Y)&=n\beta(D,Y),\\
\beta(D,nY)&=\frac{1}{n}\beta(D,Y).
\end{align*}

In particular, we can use these properties to canonically extend the definition of $\beta(D,E)$ to $\mathbb{Q}$-divisors $D$ and $E$ (with $D$ big). We also show that $\beta(D,E)$ depends only on the numerical equivalence class of $D$ and $E$.
\begin{lemma}
Let $D,D',E,E'$ be nonzero effective Cartier divisors on a projective variety $X$ over $K$ with $D$ and $D'$ big. Suppose that $D\equiv D'$ and $E\equiv E'$. Then
\begin{align*}
\beta(D,E)=\beta(D',E').
\end{align*}
\end{lemma}

\begin{proof}
We prove that $\beta(D,E)=\beta(D,E')$ (the proof that $\beta(D,E)=\beta(D',E)$ being similar). Let $n=\dim X$. Since $E\equiv E'$ and $D$ is big, for any positive rational $\varepsilon$, the $\mathbb{Q}$-divisor $E-E'+\varepsilon D$ is big. It follows that for a sufficiently large positive integer $N$, $h^0(N(E-E'+\varepsilon D))>0$, and so $E+\varepsilon D\sim E'+F$ for some $\mathbb{Q}$-divisor $F$. Therefore,
\begin{align*}
\beta(D,E')\geq \beta(D,E'+F)=\beta(D,E+\varepsilon D). 
\end{align*}
Fix an ample divisor $A$ on $X$ and let $c>\frac{A^{n-1}.D}{A^{n-1}.E}$ be a positive integer. Then $(ND-NmE).A^{n-1}<0$ for all $m\geq c$, and so $h^0(ND-NmE)=0$ for all $m\geq c$.  Let $M$ be a positive integer such that $M(E+\varepsilon D)$ is an integral Cartier divisor and let $N$ be a positive integer such that $M|N$.
Then
\begin{align*}
\sum_{m=1}^\infty h^0(ND-mM(E+\varepsilon D))&=\sum_{m=1}^{cN/M} h^0(ND-mM(E+\varepsilon D))\\
&\geq \sum_{m=1}^{cN/M} h^0((N-c\varepsilon N)D-mME)\\
&=\sum_{m=1}^{\infty} h^0(N(1-c\varepsilon)D-mME),
\end{align*}
and so
\begin{align*}
\frac{\sum_{m=1}^\infty h^0(ND-mM(E+\varepsilon D))}{Nh^0(ND)}&\geq \frac{\sum_{m=1}^{\infty} h^0(N(1-c\varepsilon)D-mME)}{Nh^0(N(1-c\varepsilon)D)}\frac{h^0(N(1-c\varepsilon)D)}{h^0(ND)}.
\end{align*}
Taking limits (see Definition \ref{voldef} and Remark \ref{volrem}) gives
\begin{align*}
\beta(D,E+\varepsilon D)=M\beta(D,M(E+\varepsilon D))&\geq M\beta((1-c\varepsilon)D,ME)\frac{\vol((1-c\varepsilon)D)}{\vol(D)}\\
&\geq (1-c\varepsilon)^{n+1}\beta(D,E).
\end{align*}
Thus,
\begin{align*}
\beta(D,E')\geq (1-c\varepsilon)^{n+1}\beta(D,E). 
\end{align*}
Since $\varepsilon>0$ was arbitrary, we find that $\beta(D,E')\geq \beta(D,E)$. Interchanging the roles of $E$ and $E'$ yields that $\beta(D,E)=\beta(D,E')$ 
\end{proof}

We now prove some useful inequalities involving the beta constant.

\begin{lemma}
\label{betalcm}
Let $Y$ be a closed subscheme of a projective variety $X$. Let $D,E$ be effective Cartier divisors on $X$ intersecting properly and suppose that $mY\subset D$ and $nY\subset E$ (as closed subschemes). Let $A$ be a big Cartier divisor on $X$. Then
\begin{align*}
\beta(A,Y)\geq m\beta(A,D)+n\beta(A,E).
\end{align*}
In particular,
\begin{align*}
\beta(A,D\cap E)\geq \beta(A,D)+\beta(A,E).
\end{align*}
\end{lemma}

\begin{proof}
We first note that if $(x_1,\ldots, x_n)$ is a regular sequence, then for any positive integers $i_1,\ldots, i_n$, $(x_1^{i_1},\ldots, x_n^{i_n})$ is again a regular sequence \cite[Theorem 26]{Matsumura}.  This implies that if $i$ and $j$ are positive integers, then $iD$ and $iE$ intersect properly and $\cO_X(-iD)\cap \cO_X(-jE)= \cO_X(-iD-jE)$. 

Let $N>0$. We consider the two filtrations on $H^0(X,\cO(NA))$ given by order of vanishing along $D$ and $E$, respectively:
\begin{align*}
H^0(X,\cO(NA))\supset H^0(X,\cO(NA-D))\supset  H^0(X,\cO(NA-2D))\supset \cdots\\
H^0(X,\cO(NA))\supset H^0(X,\cO(NA-E))\supset H^0(X,\cO(NA-2E))\supset \cdots.
\end{align*}

Let $ s_1,\dots, s_{h^0(NA)}$ be a basis of $H^0(X,\cO(NA))$ that is adapted to both filtrations. Let $\mathcal{I}$ denote the sheaf of ideals associated to $Y$.

For a nonzero section $s\in H^0(X,\cO(NA))$, let $\mu_D(s), \mu_E(s)$, and $\mu_Y(s)$, respectively, be the largest nonnegative integer $\mu$ such that $s\in H^0(X,\cO(NA-\mu D)), s\in H^0(X,\cO(NA-\mu E))$, and $s\in H^0(X,\cO(NA)\otimes \mathcal{I}^{\mu})$, respectively. Since  $s_1,\dots, s_{h^0(NA)}$ is adapted to the filtrations above, one easily finds
\begin{align*}
\sum_{j=1}^{\infty} h^0(\cO(NA-jD))&=\sum_{l=1}^{h^0(NA)} \mu_D(s_l),\\
\sum_{j=1}^{\infty} h^0(\cO(NA-jE))&=\sum_{l=1}^{h^0(NA)} \mu_E(s_l).
\end{align*}
We also have the inequality
\begin{align*}
\sum_{j=1}^{\infty}h^0(\cO(NA)\otimes \mathcal{I}^j)\geq \sum_{l=1}^{h^0(NA)}\mu_Y(s_l).
\end{align*}
 Since $D$ and $E$ intersect properly, if $s_l\in H^0(X,\cO(NA-iD))$ and $s_l\in H^0(X,\cO(NA-jE))$, then $s_l\in H^0(X,\cO(NA-iD-jE))\subset H^0(X,\cO(NA)\otimes \mathcal{I}^{im+jn})$. Therefore, $\mu_Y(s_l)\geq m\mu_D(s_l)+n\mu_E(s_l)$. It follows that 
\begin{align*}
\sum_{j=1}^{\infty}h^0(\cO(NA)\otimes \mathcal{I}^j)\geq m\sum_{i=1}^\infty h^0(\cO(NA-iD))+n\sum_{j=1}^\infty h^0(\cO(NA-jE)).
\end{align*}
Dividing by $Nh^0(\cO(NA))$ and taking a limit then gives the result.
\end{proof}

Before stating and proving the next inequality, we recall the definition of the Seshadri constant.

\begin{definition}
\label{Seshadridef}
    Let $X$ be a projective variety over a field $K$. Let $Y$ be a closed subscheme of $X$ and $\pi : \tilde{X} \to X$ be the blowing-up along $Y$. Let $A$ be a nef Cartier divisor on $X$. Define the Seshadri constant $\epsilon(A,Y)$ of $Y$ with respect to $A$ to be the real number 
    \begin{align*}
        \epsilon(A,Y)=\sup \{\gamma \in \bQ^{\geq 0}\ |\ \pi^*A-\gamma E\text{ is }\bQ\text{-nef}\},
    \end{align*}
    where $E$ is an effective Cartier divisor on $\tilde{X}$ whose associated invertible sheaf is the dual of $\pi^{-1}\cI_Y\cdot \cO_{\tilde{X}}$.
\end{definition}

\begin{lemma}
\label{betaseshadri}
Let $X$ be a normal projective variety of dimension $r$. Let $Y$ be a closed subscheme of $X$ with $\codim Y\geq 2$, let $D\supset Y$ be an effective Cartier divisor, and let $A$ be an ample Cartier divisor on $X$. Then
\begin{align*}
\beta(A,Y)\geq \beta(A,D)+\frac{1}{r+1}\epsilon(A,Y).
\end{align*}
\end{lemma}

\begin{proof}
Let $\pi:\tilde{X}\to X$ be the blowup of $X$ along $Y$, and let $E$ be the associated exceptional divisor. Let $\epsilon_Y'\leq \epsilon(A,Y)$ be a positive rational number and let $\delta>0$ be a rational number. By the same proof as in \cite{HL21}, for all sufficiently small $\varepsilon>0$ (depending on $\delta$) and sufficiently large and divisible $N$, we can find an effective divisor $F$ such that
\begin{enumerate}
\item $N(1+\delta)A\sim F$
\item $\pi^*F\geq N(\epsilon_Y'+\varepsilon)E$
\item $D$ and $F$ are in general position.
\end{enumerate}
As $X$ is normal, the last condition and Serre's criterion imply that $D$ and $F$ intersect properly. For all sufficiently large and divisible $N$, $\pi_*\CO_{\tilde{X}}(-N(\epsilon_Y'+\varepsilon)E)=\mathcal{I}_Y^{N(\epsilon_Y'+\varepsilon)}$ \cite{HL21} and since $X$ is normal and $\pi$ is birational, we have $\pi_*\pi^*F=F$ \cite[Lemma 6.2(a)]{Voj23}. Then $\beta(A,F)=\frac{1}{(r+1)N(1+\delta)}$, $N(\epsilon_Y'+\varepsilon)Y\subset F$ (for sufficiently large and divisible $N$), $Y\subset D$, and Lemma \ref{betalcm} gives
\begin{align*}
\beta(A,Y)\geq \beta(A,D)+N(\epsilon_Y'+\varepsilon)\frac{1}{(r+1)N(1+\delta)}=\beta(A,D)+(\epsilon_Y'+\varepsilon)\frac{1}{(r+1)(1+\delta)}.
\end{align*} 
Since we may choose $\varepsilon, \delta$, and $\epsilon_Y'-\epsilon(A,Y)$ arbitrarily small, we find that
\begin{align*}
\beta(A,Y)\geq \beta(A,D)+\frac{1}{r+1}\epsilon(A,Y)
\end{align*} 
as desired.
\end{proof}

When $Y$ is a closed subscheme supported at a point $P$, we will relate $\beta(A,Y)$, $\mu_P(Y)$, and the volume of $A$. We first recall the definition of the volume of a divisor.

\begin{definition}\label{voldef}
Let $X$ be a projective variety of dimension $n$, and let $D$ be a Cartier divisor on $X$. The
volume of $D$ is defined to be the non-negative real number
\begin{align*}
    \emph{vol}(D)=\limsup_{m \to \infty}\frac{h^0(X,mD)}{m^n/n!}.
\end{align*}
\end{definition}

\begin{remark}\label{volrem}
The $\limsup$ in Definition \ref{voldef} can be replaced by a limit \cite[Remark~2.2.50]{Laz04}, and the definition can be extended to $\mathbb{Q}$-divisors \cite[Remark~2.2.39]{Laz04}.
\end{remark}

\begin{remark}
We have $\mathrm{vol}(D) > 0$ if and only if $D$ is big. 
\end{remark}

\begin{remark}
If $D$ is nef, then $\emph{vol}(D)=D^n$ is the top self-intersection of $D$.
\end{remark}

We need the following lemma.

\begin{lemma}
\label{multlemma1}
Let $A$ be a big Cartier divisor on a projective variety $X$ of dimension $r$. Let $Y$ be a closed subscheme of $X$ supported at a point $P$ of $X$, let $\mathcal{I}$ be the associated ideal sheaf, and let $\mu_P$ be the multiplicity of $Y$ at $P$. 
Then
\begin{align}
\label{RRAI}
\dim H^0(X,\cO(nA)\otimes \mathcal{I}^m)\geq h^0(X,\cO(nA))-\mu_P\frac{m^r}{r!}+O(m^{r-1}).
\end{align}
In particular, let $c\in\mathbb{Q}$ be such that
\begin{align*}
0<c<\sqrt[r]{\frac{\emph{vol}(A)}{\mu_P}}.
\end{align*}
Then for $n\gg 0$ there exists an effective divisor $F$ such that 
\begin{enumerate}
\item $F\sim nA$,
\item $\lfloor cn \rfloor Y\subset F$ (as closed subschemes).
\end{enumerate}
\end{lemma}

\begin{proof}
From the definition of the multiplicity $\mu_P$, we have
\begin{align}
\label{RR1}
\dim \cO_P/\mathcal{I}_P^m=\mu_P \frac{m^r}{r!}+O(m^{r-1}).
\end{align}

By definition of the ideal sheaf, we have an exact sequence
\begin{align*}
0\to \mathcal{I}^{m}\to \cO_X\to i_*\cO_{mY}\to 0,
\end{align*}
where $i:mY\to X$ is the inclusion map. Tensoring with $\cO(nA)$, we have an exact sequence
\begin{align*}
0\to \cO(nA)\otimes \mathcal{I}^{m}\to \cO(nA)\to \cO(nA)\otimes i_*\cO_{mY}= i_*\cO_{mY}\to 0,
\end{align*}
where the equality $\cO(nA)\otimes i_*\cO_{mY}= i_*\cO_{mY}$ follows easily from the fact that $\cO(nA)$ is locally free of rank $1$ and $i_*\cO_{mY}$ is supported at a point. Taking global sections, we find
\begin{align*}
0\to H^0(X,\cO(nA)\otimes \mathcal{I}^{m})\to H^0(X,\cO(nA)) &\to H^0(X,i_*\cO_{mY})\\
&= H^0 (mY,\cO_{mY}) = \cO_P/\mathcal{I}_P^{m}.
\end{align*}
It follows that
\begin{align*}
\dim H^0(X,\cO(nA)\otimes \mathcal{I}^{m })\geq \dim H^0(X,\cO(nA)) - \dim \cO_P/\mathcal{I}_P^{m},
\end{align*}
and then \eqref{RRAI} follows from \eqref{RR1}.

From \eqref{RRAI}, the definition of the volume, and Remark \ref{volrem}, it follows that if $c<\sqrt[r]{\frac{\vol(A)}{\mu_P}}$, then there exists a nonzero global section $s$ of $\cO(nA)\otimes \mathcal{I}^{ \lfloor cn\rfloor}$ for $n\gg 0$, and taking $F=\dv(s)$ gives the desired divisor.
\end{proof}

We now prove a basic inequality between $\beta(A,Y)$, $\mu_P(Y)$, and $\vol(A)$.

\begin{lemma}
\label{betamultsimple}
Let $A$ be a big Cartier divisor on a projective variety $X$ of dimension $r$. Let $Y$ be a closed subscheme of $X$ supported at a point $P$ of $X$, and let $\mu_P$ be the multiplicity of $Y$ at $P$. 
Then
\begin{align*}
\beta(A,Y)\geq \frac{r}{r+1}\sqrt[r]{\frac{\vol(A)}{\mu_P}}.
\end{align*}

\end{lemma}

\begin{proof}
Let $c\in\mathbb{Q}$ be such that
\begin{align*}
0<c<\sqrt[r]{\frac{\vol(A)}{\mu_P}}.
\end{align*}
Let $N$ be a positive integer such that $cN$ is an integer. Let $\mathcal{I}$ be the sheaf of ideals associated to $Y$.  Then we calculate
\begin{align*}
\sum_{m=1}^{cN}h^0(\cO(NA)\otimes \mathcal{I}^m)&\geq \sum_{m=1}^{cN} \left(h^0(\cO(NA))-\mu_p\frac{m^r}{r!}+O(N^{r-1})\right)\\
&\geq cNh^0(\cO(NA))-\mu_P\frac{(cN)^{r+1}}{(r+1)!}+O(N^r).
\end{align*}
Dividing by $Nh^0(\cO(NA))$ we find that
\begin{align*}
\frac{\sum_{m=1}^{cN}h^0(\cO(NA)\otimes \mathcal{I}^m)}{Nh^0(\cO(NA))}&\geq c-\mu_P\frac{c^{r+1}N^{r}}{(r+1)!h^0(\cO(NA)}+O\left(\frac{N^{r-1}}{h^0(\cO(NA)}\right).
\end{align*}
Letting $N\to\infty$ gives
\begin{align*}
\beta(A,Y)\geq  \frac{c}{(r+1)\vol(A)}\left((r+1)\vol(A)-\mu_Pc^r\right).
\end{align*}
Since this is true for all appropriate values of $c$, letting $c\to \sqrt[r]{\frac{\vol(A)}{\mu_P}}$  gives the inequality.
\end{proof}

Under more hypotheses, we give another estimate involving $\beta(A,Y)$, $\mu_P(Y)$, and $\vol(A)$. Towards this end, we first prove the following lemma.

\begin{lemma}
\label{multlemma2}
Let $A$ and $D$ be big Cartier divisors on a normal projective surface $X$, with $D$ a positive multiple of a prime Cartier divisor. Let $Y$ be a closed subscheme of $X$ supported at a point $P$, and let $\mu_P$ be the multiplicity of $Y$ at $P$. Suppose that $D\supset Y$ is a regular chain of closed subschemes of $X$ and
\begin{align*}
\mu_P< \vol(D).
\end{align*}

Let $c\in\mathbb{Q}$ be such that
\begin{align*}
c<\sqrt{\frac{\vol(A)}{\mu_P}}.
\end{align*}

Then for $n\gg 0$ there exists an effective divisor $F$ such that 
\begin{enumerate}
\item $F\sim nA$,
\item $\lfloor cn\rfloor Y\subset F$ (as closed subschemes),
\item $D$ and $F$ intersect properly.
\end{enumerate}

\end{lemma}

\begin{proof}
For simplicity and ease of notation, we assume throughout that $cn\in\mathbb{Z}$. Let $D=mD_0$ where $D_0$ is a prime Cartier divisor on $X$. Since $\mu_P<\vol(D)$, applying Lemma \ref{multlemma1} with $A=D$ and an appropriate rational number $c'>1$, it follows that for any sufficiently large and divisible integer $n'$, there exists an effective divisor $F'$ such that $F'\sim n'D$ and $n'c'Y\subset F'$. Since $X$ is normal, we can write $F'=iD_0+D'$, where $i$ is a nonnegative integer, $D'$ is effective, and $\Supp D_0\not\subset \Supp D'$.  By Lemma \ref{orderlemma},
\begin{align*}
n_Y(D')\geq n_Y(F')-n_Y(iD_0)-1\geq n'c'-i/m-1.
\end{align*}
Let $n_0=mn'-i$, and note also that $D'\sim n_0D_0$. Clearly $i\leq mn'$, and in fact $i<mn'$ for sufficiently large $n'$ (or equivalently, $n_0>0$, as we now assume) since $c'>1$ and therefore $n_Y(D')\geq n'c'-i/m-1\geq n'(c'-1)-1>0$ for large $n'$. Choosing $n'$ sufficiently large, we also have $n_Y(D')> n'-i/m=n_0/m$. 

By Lemma \ref{multlemma1} again, there exists an effective divisor $F''$ such that $F''\sim nA$ and $cnY\subset F''$. Write $F''=jD_0+D''$, where $j$ is a nonnegative integer, $D''$ is effective, and $\Supp D_0\not\subset \Supp D''$. Replacing $F''$, $n$, $j$, and $D''$ by suitable multiples, we may assume that $n_0|j$. Then $F''=jD_0+D''\sim \frac{j}{n_0}D'+D''$. By the same argument as before,
\begin{align*}
n_Y(D'')\geq n_Y(F'')-n_Y(jD_0)-1\geq cn-j/m-1.
\end{align*}
Let $F=\frac{j}{n_0}D'+D''$. Then
\begin{align*}
n_Y(F)\geq \frac{j}{n_0}n_Y(D')+n_Y(D'')\geq \frac{j}{n_0}\frac{n_0}{m}+ cn-j/m-1=cn-1.
\end{align*}
Replacing $c$ by a slightly large constant and taking $n$ sufficiently large then gives
\begin{align*}
n_Y(F)\geq cn.
\end{align*}

It follows that $cnY\subset F$, and we end by noting that $F\sim nA$ and $\Supp D\not\subset \Supp F=\Supp D'\cup \Supp D''$, which implies that $D$ and $F$ are in general position, and hence intersect properly (as a normal surface is Cohen-Macaulay).
\end{proof}

We prove the following inequality, which will be used in Corollary \ref{mainmultineq}. For simplicity, we state it under an ampleness assumption.

\begin{lemma}
\label{betamultineq}
Let $A$ and $D$ be ample Cartier divisors on a normal projective surface $X$, with $D$ a positive multiple of a prime Cartier divisor. Let $Y$ be a closed subscheme of $X$ supported at a point $P$, and let $\mu_P$ be the multiplicity of $Y$ at $P$. Suppose that $D\supset Y$ is a regular chain of closed subschemes of $X$ and
\begin{align*}
\mu_P<D^2.
\end{align*}
Then
\begin{align}
\label{betaYDineq}
\beta(A,Y)\geq \left(1+\sqrt{\frac{D^2}{\mu_P}}\right)\beta(A,D)
\end{align}
and
\begin{align*}
\beta(A,Y)-\beta(A,D)&\geq \frac{2}{3} \sqrt{\frac{A^2}{\mu_P}}\frac{\sqrt{\frac{D^2}{\mu_P}}}{1+\sqrt{\frac{D^2}{\mu_P}}}\\
&> \frac{1}{3} \sqrt{\frac{A^2}{\mu_P}}.
\end{align*}
\end{lemma}

\begin{proof}
Let $c\in\mathbb{Q}$ with $0<c<\sqrt{\frac{
D^2}{\mu_P}}$. Then by Lemma \ref{multlemma2}, for $n$ sufficiently large with $nc\in\mathbb{Z}$, we can find an effective divisor $F$ such that $F\sim  nD$, $ncY\subset F$, and $D$ and $F$ intersect properly. Note that $\beta(A,F)=\beta(A,nD)=\frac{1}{n}\beta(A,D)$. Then by Lemma \ref{betalcm}, we have
\begin{align*}
\beta(A,Y)\geq \beta(A,D)+cn\beta(A,F)=\beta(A,D)+c\beta(A,D).
\end{align*}
Since this holds for every positive rational $c<\sqrt{\frac{D^2}{\mu_P}}$, the inequality \eqref{betaYDineq} holds. This implies
\begin{align*}
\beta(A,Y)-\beta(A,D)&\geq  \beta(A,Y)- \frac{1}{1+\sqrt{\frac{D^2}{\mu_P}}}\beta(A,Y)\\
&\geq \frac{\sqrt{\frac{D^2}{\mu_P}}}{ 1+\sqrt{\frac{D^2}{\mu_P}}}\beta(A,Y)\\
&\geq \frac{2}{3} \sqrt{\frac{A^2}{\mu_P}}\frac{\sqrt{\frac{D^2}{\mu_P}}}{ 1+\sqrt{\frac{D^2}{\mu_P}}}\\
&>\frac{1}{3} \sqrt{\frac{A^2}{\mu_P}},
\end{align*}
where the last line uses that $\mu_P<D^2$.
\end{proof}

Finally, we prove a useful estimate for certain beta constants on the blowup of a projective variety $X$ of dimension $n$ at a point (the idea is to gives a slight improvement to the equality $\beta(D,D)=\frac{1}{n+1}$ for an ample divisor $D$ on $X$).

\begin{lemma}
\label{exclemma}
Let $D$ be an ample effective Cartier divisor on a projective variety $X$ over $K$ of dimension $n$. Let $P\in X(K)$, and let $\pi:\tilde{X}\to X$ be the blowup at $P$, with associated exceptional divisor $E$.  Then for all sufficiently small positive $\delta\in\mathbb{Q}$,
\begin{align}
\label{exceqn}
\beta(\pi^*D-\delta E, \pi^*D-E)>\frac{1}{n+1}.
\end{align}
\end{lemma}

\begin{proof}
First, we recall that for some positive $\delta'\in\mathbb{Q}$, the $\mathbb{Q}$-divisor $\pi^*D-\delta E$ is ample for all $0<\delta<\delta'$, $\delta\in\mathbb{Q}$ \cite[II, Prop.~7.10(b)]{Harts}. Let $\delta\in\mathbb{Q}$ satisfy $0<\delta<\min\{\delta',1\}$.   Since $\pi^*D-\delta E$ is ample, 
\begin{align*}
(\pi^*D-\delta E)^{n-1}.E=(-\delta)^{n-1}E^n>0.
\end{align*}
Let
\begin{align*}
\gamma:=(-1)^{n-1}E^n>0.
\end{align*}

Then
\begin{align*}
(a\pi^*D-bE)^n&=a^n(\pi^*D)^n+(-1)^nb^nE^n\\
&=a^nD^n-\gamma b^n.
\end{align*}

Let $N$ be such that $N\delta\in\mathbb{Z}$. We need to estimate
\begin{align*}
h^0\left(\tilde{X},N(\pi^*D-\delta E)-m(\pi^*D-E)\right)
=h^0\left(\tilde{X},(N-m)\pi^*D-(N\delta-m) E\right)
\end{align*}
for $m>0$. We break the estimate into two cases. We first consider $0<m\leq N\delta$. Since
\begin{align*}
\frac{N\delta-m}{N-m}<\delta<\delta',
\end{align*}
$(N-m)\pi^*D-(N\delta-m) E$ is ample, and it follows from Riemann-Roch estimates (e.g., using that $\pi^*D$ and $\pi^*D-\delta E$ are $\mathbb{Q}$-nef; see \cite[p.~233, Cas $k\leq n$]{Aut09}) that
\begin{align*}
h^0\left((N-m)\pi^*D-(N\delta-m) E\right)&=\frac{((N-m)\pi^*D-(N\delta-m) E)^n}{n!}+O(N^{n-1})\\
&=\frac{(N-m)^nD^n-\gamma(N\delta-m)^n}{n!}+O(N^{n-1})
\end{align*}
for $0<m\leq N\delta$. Next, for $N\delta<m<N$, note that
\begin{align*}
(N-m)\pi^*D-(N\delta-m) E\geq (N-m)\pi^*D,
\end{align*}
and so for $N\delta<m<N$,
\begin{align*}
h^0\left((N-m)\pi^*D-(N\delta-m) E\right)
\geq h^0((N-m)\pi^*D)=\frac{(N-m)^nD^n}{n!}+O(N^{n-1}).
\end{align*}

Therefore,
\begin{multline*}
\sum_{m\geq 1}h^0\left (N(\pi^*D-\delta E)-m(\pi^*D-E) \right)\\
\geq \sum_{m=1}^{N\delta}\frac{(N-m)^nD^n-\gamma(N\delta-m)^n}{n!}+\sum_{m=N\delta+1}^N\frac{(N-m)^nD^n}{n!}+O(N^n).
\end{multline*}
By standard formulas,
\begin{align*}
\sum_{m=1}^{N\delta}\frac{(N-m)^nD^n-\gamma(N\delta-m)^n}{n!}=\frac{N^{n+1}}{(n+1)!}(D^n-(1-\delta)^{n+1}D^n-\gamma\delta^{n+1})+O(N^n)
\end{align*}
and
\begin{align*}
\sum_{m=N\delta+1}^N\frac{(N-m)^nD^n}{n!}=\frac{N^{n+1}}{(n+1)!}(1-\delta)^{n+1}D^n+  O(N^n).
\end{align*}
Therefore,
\begin{align*}
\sum_{m\geq 1}h^0(N(\pi^*D-\delta E)-m(\pi^*D-E))\geq \frac{N^{n+1}}{(n+1)!}(D^n-\gamma\delta^{n+1})+O(N^n).
\end{align*}
On the other hand,
\begin{align*}
Nh^0(N(\pi^*D-\delta E))&=N \left(\frac{N^n}{n!}(D^n-\gamma\delta^n)+O(N^{n-1})\right)\\
&=\frac{N^{n+1}}{n!}(D^n-\gamma\delta^n)+O(N^n).
\end{align*}
Note also that, in particular, $\gamma\delta^n<D^n$.  Putting things together, we find
\begin{align*}
\lim_{N\to\infty}\frac{\sum_{m\geq 1}h^0(N(\pi^*D-\delta E)-m(\pi^*D-E))}{Nh^0(N(\pi^*D-\delta E))}&\geq \frac{1}{n+1}\frac{D^n-\gamma\delta^{n+1}}{D^n-\gamma\delta^{n}}\\
&>\frac{1}{n+1}
\end{align*}
as $\delta<1$.
\end{proof}

\section{Main Result and Some Consequences}\label{sec: main}

We now prove (and restate) the Main Theorem.

\begin{theorem}
Let $X$ be a projective surface defined over a number field $K$. Let $S$ be a finite set of places of $K$. For each $v\in S$, let $D_v\supset Y_v$ be a regular chain of nonempty closed subschemes of $X$. Let $A$ be a big Cartier divisor on $X$, and let $\varepsilon > 0$. 
Then there exists a proper Zariski-closed subset $Z\subset X$ such that for all 
$P\in X(K)\setminus Z$, 
we have 
\begin{equation*}
    \begin{aligned}
        \sum_{v\in S} \bigg(\beta(A,D_v) 
        \lambda_{D_v,v}(P) + 
       \left( \beta(A, Y_v) - \beta(A, D_v) \right) \lambda_{Y_v,v}(P) \bigg) < (1 + \varepsilon) h_A(P).
    \end{aligned}
\end{equation*}
\end{theorem}

\begin{proof}
From the regular chain assumption, for all $v\in S$, $D_v$ is an effective Cartier divisor and $\dim Y_v=0$. We may further reduce to the case that $Y_v$ is supported at a single point: after replacing $K$ by a finite extension $L$ and replacing $S$ by the set of places of $L$ lying above $S$, we may assume that every point in the support of $Y_v$ is $K$-rational. Then we may write $Y_v=Y_{1,v}+\cdots Y_{r,v}$, where $Y_{1,v},\ldots, Y_{r,v}$ are closed subschemes supported at distinct $K$-rational points. Now the reduction follows from observing that $\beta(A,Y_{i,v})\geq \beta(A, Y_v)$ for all $i$, and $\lambda_{Y_v,v}(P)=\max_i \lambda_{Y_{i,v},v}(P)+O(1)$ for all $P\in X(K)\setminus Y_v$ (see Remark \ref{remzerodim}).

Let $N>0$ and let $v\in S$. We consider the two filtrations on $H^0(X,\cO(NA))$ given by order of vanishing along $D_v$ and $Y_v$, respectively:
\begin{align*}
H^0(X,\cO(NA))\supset H^0(X,\cO(NA-D_v))\supset H^0(X,\cO(NA-2D_v))\supset \cdots
\end{align*}
and
\begin{align*}
H^0(X,\cO(NA))\supset H^0(X,\cO(NA)\otimes \mathcal{I}_v)\supset H^0(X,\cO(NA)\otimes \mathcal{I}_v^2)\supset \cdots.
\end{align*}

By Lemma \ref{filtrationlemma}, there exists a basis $s_{1,v},\ldots, s_{h^0(NA),v}$ of $H^0(X,\cO(NA))$ adapted to both filtrations. It follows that 

\begin{align*}
\sum_{i=1}^{h^0(NA)} \mathrm{div}(s_{i,v})&\geq \sum_{i=1}^\infty  \left(h^0(NA-iD_v)-h^0(NA-(i+1)D_v)\right) iD_v\\
&\geq \left(\sum_{i=1}^\infty h^0(NA-iD_v)\right) D_v
\end{align*}
and similarly, 
$$\sum_{i=1}^{h^0(NA)} \mathrm{div}(s_{i,v}) \supset \left(\sum_{i=1}^\infty h^0(\cO(NA)\otimes \mathcal{I}_v^i)\right) Y_v. $$
as closed subschemes. Then we can write
\begin{align*}
\sum_{i=1}^{h^0(NA)} \mathrm{div}(s_{i,v})=\left(\sum_{i=1}^\infty h^0(NA-iD_v)\right) D_v+F_v
\end{align*}
for some effective divisor $F_v$. By Lemma \ref{orderlemma},
\begin{align*}
n_{Y_v}(F_v)&\geq n_{Y_v}\left(\sum_{i=1}^{h^0(NA)} \mathrm{div}(s_{i,v})\right)-\left(\sum_{i=1}^\infty h^0(NA - iD_v)\right)-1\\
&\geq  \left(\sum_{i=1}^\infty h^0(\cO(NA)\otimes \mathcal{I}_v^i)\right)-\left(\sum_{i=1}^\infty h^0(NA - iD_v)\right)-1.\notag
\end{align*}
It follows that
\begin{align*}
\lambda_{F_v,v}(P) &\geq n_{Y_v}(F_v)\lambda_{Y_v,v}(P)+O(1)\\
&\geq  \left(\sum_{i=1}^\infty h^0(\cO(NA)\otimes \mathcal{I}_v^i)-\sum_{i=1}^\infty h^0(NA - iD_v)-1\right)\lambda_{Y_v,v}(P)+O(1)
\end{align*}
for $P\in X(K)\setminus \Supp F_v$.

We now apply the Subspace Theorem in the form of \cite[Theorem 2.10]{RV20}, with the union of the bases $s_{1,v},\ldots, s_{h^0(NA),v}, v\in S,$ of $H^0(X,\CO(NA))$ constructed above;  we obtain that for $\varepsilon>0$, there exists a proper Zariski-closed subset $Z$ of $X$ such that, up to $O(1)$,
\begin{align*}
\sum_{v\in S} \left(\sum_{i=1}^\infty h^0(NA-iD_v)\right) \lambda_{D_v,v}(P)
+&\left(\sum_{i=1}^\infty h^0(\cO(NA)\otimes \mathcal{I}_v^i)-\sum_{i=1}^\infty h^0(NA - iD_v)-1\right)\lambda_{Y_v,v}(P) \\
&\leq\sum_{v\in S}\left(\sum_{i=1}^\infty h^0(NA-iD_v)\right) \lambda_{D_v,v}(P)+\lambda_{F_v,v}(P)\\
&\leq (h^0(NA)+\varepsilon)h_{NA}(P)\\
&\leq(Nh^0(NA)+N\varepsilon)h_{A}(P)
\end{align*}
for all $P\in X(K)\setminus Z$. Dividing by $Nh^0(NA)$ and taking $N$ sufficiently large gives the desired inequality. 
\end{proof}

We show that Theorem \ref{BetaMain} implies the Ru-Vojta inequality \cite{RV20} for surfaces (a similar, but more complicated argument, could be used to derive Theorem \ref{theoremRV} in the case of surfaces).

\begin{corollary}
\label{RVsurfaces}
    Let $X$ be a projective surface over a number field $K$, and let $D_1,\ldots, D_q$ be nonzero effective Cartier divisors intersecting properly on $X$. Let $\mathscr{L}$ be a big line bundle on $X$. Let $S$ be a finite set of places of $K$ and let $\varepsilon>0$. Then there exists a proper Zariski-closed subset $Z\subset X$ such that for all 
$P\in X(K)\setminus Z$, we have 
    \begin{equation*}
        \sum_{i=1}^q \beta(\mathscr{L},D_i) m_{D_i,S}(P) \leq (1+\varepsilon) h_{\mathscr{L}}(P).
    \end{equation*}
\end{corollary}

\begin{proof}
For any point $P\in X(K)$, we have
\begin{equation*}
\sum_{i=1}^q \beta(\mathscr{L},D_i) m_{D_i,S}(P) \leq 
\sum_{v\in S} 
\bigg(\beta(\mathscr{L},D_{P,1,v}) \lambda_{D_{P,1,v},v}(P)+\beta(\mathscr{L},D_{P,2,v}) \lambda_{D_{P,2,v},v}(P)+O(1) \bigg),
\end{equation*}
for some divisors $D_{P,1,v}$ and $D_{P,2,v}$ depending on $P$ and $v$ (but with the $O(1)$ term independent of $P$). Then by considering the finitely many choices for $D_{P,1,v}$ and $D_{P,2,v}$, it suffices to study the sum
\begin{align*}
\sum_{v\in S} \bigg(\beta(\mathscr{L},D_{1,v}) \lambda_{D_{1,v},v}(P)+\beta(\mathscr{L},D_{2,v}) \lambda_{D_{2,v},v}(P)\bigg),
\end{align*}
where for each $v\in S$ the divisors $D_{1,v}, D_{2,v}$ intersect properly. Again, by considering finitely many cases we may assume that $\lambda_{D_{1,v},v}(P)\geq \lambda_{D_{2,v},v}(P)$ for all $v\in S$ and thus
\begin{align*}
\lambda_{D_{2,v},v}(P)=\min\{\lambda_{D_{1,v},v}(P),\lambda_{D_{2,v},v}(P)\}=\lambda_{D_{1,v}\cap D_{2,v},v}(P)+O(1)
\end{align*}
for all $v\in S$.  By Lemma \ref{betalcm},
\begin{align*}
\beta(\mathscr{L},D_{1,v}\cap D_{2,v})\geq \beta(\mathscr{L},D_{1,v})+\beta(\mathscr{L},D_{2,v}).
\end{align*}
Therefore,  by Theorem \ref{BetaMain}, up to $O(1)$,
\begin{align*}
&\sum_{v\in S} \bigg(\beta(\mathscr{L},D_{1,v}) \lambda_{D_{1,v},v}(P)+\beta(\mathscr{L},D_{2,v}) \lambda_{D_{2,v},v}(P) \bigg)\\
=&\sum_{v\in S} \bigg( \beta(\mathscr{L},D_{1,v}) \lambda_{D_{1,v},v}(P)+\beta(\mathscr{L},D_{2,v}) \lambda_{D_{1,v}\cap D_{2,v},v}(P)\bigg)\\
\leq & \sum_{v\in S} \bigg(\beta(\mathscr{L},D_{1,v}) \lambda_{D_{1,v},v}(P)+(\beta(\mathscr{L},D_{1,v}\cap D_{2,v})- \beta(\mathscr{L},D_{1,v}))\lambda_{D_{1,v}\cap D_{2,v},v}(P) \bigg)\\
\leq &~ (1+\varepsilon) h_{\mathscr{L}}(P)
\end{align*}
for all $K$-rational points outside a proper Zariski-closed subset $Z$ of $X$. Finally, we may omit any $O(1)$ term by enlarging $Z$.
\end{proof}

From \cite{HL21}, if $X$ is a nonsingular projective variety of dimension $n$ and $D$ is a divisor on $X$, then $$\beta(A,D) \geq \frac{1}{n+1}\epsilon(A,D), $$ and in particular, on a surface $X$,
\begin{align}\label{bsineq}
\beta(A,D) \geq \frac{1}{3}\epsilon(A,D).
\end{align}

Using inequality (\ref{bsineq}) and Lemma \ref{betaseshadri}, we immediately derive from Theorem \ref{BetaMain} the inequality of Heier and the second author \cite{HL21} for surfaces, under a smoothness and regular chain assumption:

\begin{corollary}
Let $X$ be a nonsingular projective surface defined over a number field $K$.  Let $S$ be a finite set of places of $K$. For each $v\in S$, let $D_v\supset Y_v$ be a regular chain of nonempty closed subschemes of $X$. Let $A$ be an ample Cartier divisor on $X$, and let $\varepsilon > 0$. 
Then there exists a proper Zariski-closed subset $Z\subset X$ such that for all 
$P\in X(K)\setminus Z$, 
we have 
\begin{equation*}
    \begin{aligned}
        \sum_{v\in S}\bigg( \epsilon(A,D_v) 
        \lambda_{D_v,v}(P) + 
        \epsilon(A, Y_v)\lambda_{Y_v,v}(P) \bigg)
        < (3 + \varepsilon) h_A(P).
    \end{aligned}
\end{equation*}
\end{corollary}

Under more hypotheses, as an immediate consequence of Lemma \ref{betamultineq}, we can replace $\epsilon(A, Y_v)$ by a quantity depending only on self-intersection numbers and the multiplicity of $Y_v$ at a point. 

\begin{corollary}\label{mainmultineq}
Let $X$ be a normal projective surface defined over a number field $K$.  Let $S$ be a finite set of places of $K$. For each $v\in S$, let $D_v\supset Y_v$ be a regular chain of closed subschemes of $X$ such that $D_v$ is a positive multiple of an ample prime Cartier divisor, $Y_v$ is supported at a point $Q_v\in X(K)$, and 
\begin{align*}
\mu_v=\mu_{Q_v}(Y_v)<D_v^2.
\end{align*}
Let $A$ be an ample divisor on $X$, and let $\varepsilon > 0$. Then there exists a proper Zariski-closed subset $Z\subset X$ such that for all $P\in X(K)\setminus Z$, we have 
\begin{align*}
\sum_{v\in S} \left(3\beta(A,D_v)\lambda_{D_v}(P) + 2\sqrt{ \frac{A^2}{\mu_v}}\frac{\sqrt{\frac{D_v^2}{\mu_v}}}{1+\sqrt{\frac{D_v^2}{\mu_v}}} \lambda_{Y_v}(P)\right) < (3 + \varepsilon) h_A(P).
\end{align*}
In particular, there exists a proper Zariski-closed subset $Z\subset X$ such that for all $P\in X(K)\setminus Z$, we have 
\begin{align*}
\sum_{v\in S} \left(3\beta(A,D_v)\lambda_{D_v}(P) + \sqrt{ \frac{A^2}{\mu_v}} \lambda_{Y_v}(P)\right) < (3 + \varepsilon) h_A(P).
\end{align*}
\end{corollary}

\begin{remark}\label{Seshmultcomp}
For purposes of comparison, if $X$ is a nonsingular projective surface over $K$ and $Y$ is supported at a point $Q\in X(K)$, then we have the inequality \cite[Remark 2.4]{CEL}
\begin{align*}
\epsilon(A,Y)\leq \sqrt{ \frac{A^2}{\mu_Q(Y)}}.
\end{align*}
When $\mu_Q(Y)=1$ and $A^2$ is not a perfect square, this inequality has been conjectured to always be strict (more precisely, it has been conjectured that the Seshadri constant at a point is always rational).
\end{remark}

In the situation of Remark \ref{Seshmultcomp}, it follows that we have the inequalities
\begin{align*}
\beta(A,Y)\geq \frac{2}{3}\sqrt{\frac{A^2}{\mu_Q(Y)}} \geq \frac{2}{3}\epsilon(A,Y).
\end{align*}

We give a simple example where the inequalities are all strict.

\begin{example}
Let $X=\mathbb{P}^1\times \mathbb{P}^1$, let $A$ be a divisor of type $(1,1)$ on $X$, and let $Y=Q\in X(K)$ be a point (viewed as a closed subscheme with the reduced induced structure). Then elementary  computations give
\begin{align*}
\beta(A,Y)&=1,\\
\sqrt{ \frac{A^2}{\mu_Q(Y)}}&=\sqrt{2},\\
\epsilon(A,Y)&=1.
\end{align*}
Then we have strict inequalities
\begin{align*}
\beta(A,Y)>\frac{2}{3}\sqrt{ \frac{A^2}{\mu_Q(Y)}}>\frac{2}{3}\epsilon(A,Y).
\end{align*}
\end{example}

\section{Applications}\label{sec:applications}

In the remaining sections, we investigate some Diophantine applications of our main result Theorem \ref{BetaMain}. In Section \ref{sec: gcd} we prove, under suitable conditions, an inequality for the ``gcd height" $h_{D_i\cap D_j}(P)$, $i\neq j$, when $P$ is an $S$-integral point with respect to three properly intersecting numerically parallel divisors $D_1,D_2,D_3$ on a surface. This inequality may be viewed as complementary to the gcd inequalities of Bugeaud-Corvaja-Zannier \cite{BCZ}, Corvaja-Zannier \cite{CZ05}, and Wang-Yasufuku \cite{WY19} (for surfaces) who study the height $h_Q(P)$ (under the same integrality assumption on $P$) when $Q$ is {\it not} in the intersection of two of the divisors $D_i$. In Section \ref{sec: gcdapp}, we apply our gcd inequalities to study the equation $f(a^m, y) = b^n$, which was previously studied by Corvaja and Zannier \cite{CZ00}. In Section \ref{sec: nongeneral}, we study integral points on surfaces on the complement of three numerically parallel divisors with nonempty intersection. The results expand on earlier work of Corvaja and Zannier \cite{CZ06} in a similar setting. Finally, using the results on integral points of Section \ref{sec: nongeneral}, we study certain families of unit equations, proving a general result following work of Corvaja-Zannier \cite{CZ06, CZ10} and the second author \cite{Lev06}.

\subsection{Greatest Common Divisors on Surfaces}\label{sec: gcd}
 
In 2003, Bugeaud, Corvaja, and Zannier \cite{BCZ} initiated a new line of results with the following inequality involving greatest common divisors:

\begin{theorem}[Bugeaud-Corvaja-Zannier \cite{BCZ}]
\label{tBCZ}
Let $a,b\in \mathbb{Z}$ be multiplicatively independent integers.  Then for every $\varepsilon>0$, 
\begin{align}
\label{eBCZ}
\log \gcd (a^n-1,b^n-1)\leq \varepsilon n
\end{align}
for all but finitely many positive integers $n$.
\end{theorem}

The inequality \eqref{eBCZ} was subsequently generalized by Corvaja and Zannier \cite{CZ05}, allowing $a^n$ and $b^n$ to be replaced by elements $u$ and $v$, respectively, of a fixed finitely generated subgroup of $\Qbar^*$, and replacing  $u-1$ and $v-1$ by more general pairs of polynomials in $u$ and $v$. Silverman \cite{Sil05} interpreted these results in terms of heights and as a special case of Vojta's Conjecture.  The second author \cite{Lev19} further generalized these inequalities to multivariate polynomials, and Wang and Yasufuku proved the following general version of these results (see work of the first and second author \cite{HL22} for an even more general version):

\begin{theorem}[Wang-Yasufuku \cite{WY19}]\label{thm: WY}
Let $X$ be a Cohen–Macaulay variety of dimension $n$ defined over a number field $K$, and let $S$ be a finite set of places of $K$. Let $D_1,\ldots, D_{n+1}$ be effective Cartier divisors defined over $K$ and in general position. Suppose that there exists an ample Cartier divisor $A$ on $X$ and positive integers $d_1,\ldots, d_{n+1}$ such that $D_i\equiv d_iA$, $i=1,\ldots n+1$. Let $Y$ be a closed subscheme of $X$ of codimension at least $2$ that does not contain any point of the set
\begin{align}
\label{WYcond}
\bigcup_{i=1}^{n+1}\bigcap_{j\neq i}\mathrm{Supp}D_j.
\end{align}
Let $\varepsilon > 0$. Then there exists a proper Zariski-closed subset $Z$ of $X$ such that for any set $R$ of $(\sum_{i=1}^nD_i,S)$-integral points in $X(K)$, we have
\begin{align*}
h_Y (P) \le  \varepsilon h_A(P)
\end{align*}
for all points $P \in R \setminus Z$. 
\end{theorem}

When $X$ is a surface, in Theorem \ref{thm: WY} one may reduce to the case that $Y=Q$ is a point, and the condition \eqref{WYcond} is simply the requirement that $Q\not\in \cup_{i\neq j}D_i\cap D_j$. The main result of this section proves an inequality as in Theorem \ref{thm: WY}, but under the complementary condition $Q\in \cup_{i\neq j}D_i\cap D_j$, along with some mild additional hypotheses (necessary to exclude the case of lines in $\mathbb{P}^2$ where the analogous statement is false (Example \ref{gcdcounter})).

\begin{theorem}\label{generalPos}
Let $D_1,D_2,D_3$ be effective divisors intersecting properly on a projective surface $X$, all defined over a number field $K$.  Suppose that there exist positive integers $a_1,a_2,a_3$ such that $a_1D_1, a_2D_2, a_3D_3$ are all numerically equivalent to an ample divisor $D$.  Suppose that for some $i_0,j_0\in \{1,2,3\},~ i_0\neq j_0$, and for all $Q\in X(\overline{K})$, 
\begin{align}\label{beta23cond}
\beta(D,(a_{i_0}D_{i_0}\cap a_{j_0}D_{j_0})_Q)  > \frac{2}{3},
\end{align}
with $(a_{i_0}D_{i_0}\cap a_{j_0}D_{j_0})_Q$ as in Definition \ref{defzerodim}. Let $S$ be a finite set of places of $K$ containing all the archimedean ones and let $\varepsilon > 0$. Then  there exists a proper Zariski-closed subset $Z\subset X$ such that for any set $R\subset X(K)$ of $(D_1 + D_2 + D_3, S)$-integral points and all but finitely many points $P\in R\setminus Z$, we have 
\begin{align*}
h_{D_{i_0}\cap D_{j_0}}(P) 
        \le \varepsilon h_D(P). 
\end{align*}
\end{theorem}

The condition \eqref{beta23cond} can be replaced by the simpler condition that $D_{i_0}\cap D_{j_0}$ contains more than one point:

\begin{corollary}
\label{generalPosCor}
Suppose that the same hypotheses as in Theorem \ref{generalPos} hold, except that \eqref{beta23cond} is replaced by the assumption that $D_{i_0}\cap D_{j_0}$ contains more than one point (over $\overline{K})$.  Then  there exists a proper Zariski-closed subset $Z\subset X$ such that for any set $R\subset X(K)$ of $(D_1 + D_2 + D_3, S)$-integral points and all but finitely many points $P\in R\setminus Z$, we have 
\begin{align*}
h_{D_{i_0}\cap D_{j_0}}(P) 
        \le \varepsilon h_D(P). 
\end{align*}
\end{corollary}

\begin{proof}[Proof of Theorem \ref{generalPos}]
After replacing $K$ by a finite extension, we may assume that every point in the support of $D_i\cap D_j$, $i\neq j$, is $K$-rational.

We note that since $a_iD_i\equiv D$, we have $\beta(D,a_iD_i)=\frac{1}{3}$ for all $i$. Moreover, by Lemma \ref{betalcm},
for all $Q\in X(\overline{K})$ and $i\neq j$, we have
\begin{align}
\label{23ineq}
\beta(D,(a_iD_i\cap a_jD_j)_Q)\geq \beta(D,a_iD_i\cap a_jD_j)\geq \beta(D,a_iD_i)+\beta(D,a_jD_j)\geq \frac{2}{3}.
\end{align}
We let
\begin{align*}
\gamma=\min_{Q\in X(\kbar)}3\beta(D,(a_{i_0}D_{i_0}\cap a_{j_0}D_{j_0})_Q) -2.
\end{align*}
By assumption, $\gamma>0$. By Remark \ref{remzerodim}, for any point $P\in X(K), v\in S, i\neq j$, there is a point $Q\in \Supp (a_iD_i\cap a_jD_j)$ (depending on $P$, $v$, and $i$ and $j$) such that
\begin{align}\label{pointeq}
\lambda_{a_iD_i\cap a_jD_j,v}(P)=\lambda_{(a_iD_i\cap a_jD_j)_Q,v}(P)+O(1)
\end{align}
where the constant in the $O(1)$ is independent of $P$.

For $v\in S$, let $i_v,j_v\in \{1,2,3\}$, $i_v\neq j_v$. Let $\varepsilon>0$ and let
\begin{align*}
\gamma_v=
\begin{cases}
\gamma & \text{if }\{i_v,j_v\}=\{i_0,j_0\},\\
0 & \text{otherwise}.
\end{cases}
\end{align*}
By \eqref{23ineq}, \eqref{pointeq}, the definitions of $\gamma$ and $\gamma_v$, and Theorem \ref{BetaMain},  there exists a proper Zariski-closed subset $Z$ of $X$ such that up to $O(1)$ (and after multiplying by $3$)
\begin{align}
\label{DiDjineq}
\begin{split}
&~~~\sum_{v\in S} \bigg(\lambda_{a_{i_v}D_{i_v},v}(P)+(1+\gamma_v)\lambda_{a_{i_v}D_{i_v}\cap a_{j_v}D_{j_v},v}(P) \bigg)\\
&=\sum_{v\in S} \bigg(\lambda_{a_{i_v}D_{i_v},v}(P)+(1+\gamma_v)\max_Q\lambda_{(a_{i_v}D_{i_v}\cap a_{j_v}D_{j_v)_Q},v}(P)\bigg)\\
&\leq \sum_{v\in S} \bigg(\lambda_{a_{i_v}D_{i_v},v}(P)+(3\min_{Q}\beta(D,(a_{i_v}D_{i_v}\cap a_{j_v}D_{j_v})_Q)-1) \max_Q\lambda_{(a_{i_v}D_{i_v}\cap a_{j_v}D_{j_v})_Q,v}(P)\bigg)\\
&\leq (3+\varepsilon)h_D(P)
\end{split}
\end{align}
for all $P\in X(K)\setminus Z$. As there are only finitely many choices of $i_v,j_v$, we may find such a $Z$ that works for all choices of $i_v, j_v$, $v\in S$ (with $i_v\neq j_v$).

Next we note that since $D_1,D_2,D_3$ intersect properly, we have $D_1\cap D_2\cap D_3=\emptyset$, and by Theorem \ref{thm: Silverman}, for any $v\in S$, 
$$\min\{\lambda_{a_1D_1,v}(P),\lambda_{a_2D_2,v}(P),\lambda_{a_3D_3,v}(P)\}=O(1)$$
for all $P\in X(K)$. Let $R$ be a set of $(D_1 + D_2 + D_3, S)$-integral points. By definition and elementary properties of heights, for any $\varepsilon>0$,
\begin{align*}
\sum_{v\in S} \left(\lambda_{a_1D_{1},v}(P)+\lambda_{a_{2}D_{2},v}(P)+\lambda_{a_{3}D_{3},v}(P)\right)&=h_{a_1D_1}(P)+h_{a_2D_2}(P)+h_{a_3D_3}(P)+O(1)\\
&\geq (3-\varepsilon)h_D(P)+O(1),
\end{align*}
where the $O(1)$ possibly depends on $R$ (but not $P$).

Let $P\in R$. For $v\in S$, let $\{i_v,j_v,k_v\}=\{1,2,3\}$ be such that

$$
\lambda_{a_{i_v}D_{i_v},v}(P)\geq \lambda_{a_{j_v}D_{j_v},v}(P)\geq \lambda_{a_{k_v}D_{k_v},v}(P).
$$
Then
\begin{align*}
&~~~\sum_{v\in S} \left(\lambda_{a_{i_v}D_{i_v},v}(P)+\lambda_{a_{j_v}D_{j_v},v}(P)+\lambda_{a_{k_v}D_{k_v},v}(P)\right)\\
&=\sum_{v\in S} \left(\lambda_{a_{i_v}D_{i_v},v}(P)+\lambda_{a_{j_v}D_{j_v},v}(P)\right)+O(1)\\
&=\sum_{v\in S} \left(\lambda_{a_{i_v}D_{i_v},v}(P)+\min\{\lambda_{a_{i_v}D_{i_v},v}(P),\lambda_{a_{j_v}D_{j_v},v}(P)\}\right)+O(1)\\
&=\sum_{v\in S} \left(\lambda_{a_{i_v}D_{i_v},v}(P)+\lambda_{a_{i_v}D_{i_v}\cap a_{j_v}D_{j_v},v}(P)\right)+O(1)\\
&\geq (3-\varepsilon) h_D(P)+O(1).
\end{align*}

Noting that if $\{i_0,j_0\}\neq \{i_v,j_v\}$ then
\begin{align*} 
\lambda_{a_{i_0}D_{i_0}\cap a_{j_0}D_{j_0},v}(P)=O(1)
\end{align*}
and substituting into \eqref{DiDjineq}, we find that if $P\not\in Z$, then
\begin{align*}
(3-\varepsilon)h_D(P)+\gamma\sum_{v\in S}\lambda_{a_{i_0}D_{i_0}\cap a_{j_0}D_{j_0},v}(P)
&\leq \sum_{v\in S}\lambda_{a_{i_v}D_{i_v},v}(P)+(1+\gamma_v)\lambda_{a_{i_v}D_{i_v}\cap a_{j_v}D_{j_v},v}(P)+O(1)\\
&<(3+\varepsilon)h_D(P)+O(1).
\end{align*}

Therefore, if $P\not\in Z$,
\begin{align*}
\sum_{v\in S}\lambda_{a_{i_0}D_{i_0}\cap a_{j_0}D_{j_0},v}(P)<\frac{2\varepsilon}{\gamma}h_D(P)+O(1).
\end{align*}
Since $\min\{a_{i_0},a_{j_0}\}\min\{\lambda_{D_{i_0},v},\lambda_{D_{j_0},v}\}\leq \lambda_{a_{i_0}D_{i_0}\cap a_{j_0}D_{j_0},v}$, we conclude that for any $\varepsilon>0$, there exists a proper Zariski-closed subset $Z$ of $X$ such that for all $P\in R\setminus Z$,
\begin{align*}
\sum_{v \in S}\min\left(\lambda_{D_{i_0}, v}(P), \lambda_{D_{j_0}, v}(P)\right) 
        \le \varepsilon h_D(P)+O(1). 
\end{align*}
Since $R$ is a set of $(D_1+D_2+D_3,S)$-integral points, this is equivalent to
\begin{align*}
h_{D_{i_0}\cap D_{j_0}}(P)\le \varepsilon h_D(P)+O(1)
\end{align*}
for all $P\in R\setminus Z$. Finally, we note that we may remove the $O(1)$ in the inequality at the expense of excluding finitely many points of $R$, finishing the proof.
\end{proof}

We now prove the corollary.

\begin{proof}[Proof of Corollary \ref{generalPosCor}]
The local intersection multiplicity and the local Hilbert-Samuel multiplicity coincide (see \cite[Ex.~2.4.8, Ex.~7.1.10]{Fulton}):
\begin{align*}
\mu_Q=\mu_Q((a_{i_0}D_{i_0}\cap a_{j_0}D_{j_0})_Q)=(a_{i_0}D_{i_0}.a_{j_0}D_{j_0})_Q.
\end{align*}
Since $(a_{i_0}D_{i_0}.a_{j_0}D_{j_0})=D^2$ is the sum of the local intersection multiplicities, and $D_i$ and $D_j$ intersect at more than one point, we must have
\begin{align*}
\mu_Q<D^2
\end{align*}
for all $Q\in X(\kbar)$.

Then by Lemma \ref{betamultsimple}, for all $Q\in X(\kbar)$, 
\begin{align*}
\beta(D,(a_{i_0}D_{i_0}\cap a_{j_0}D_{j_0})_Q)\geq \frac{2}{3}\sqrt{\frac{D^2}{\mu_Q}}>\frac{2}{3},
\end{align*}
and the desired result follows from Theorem \ref{generalPos}.
\end{proof}

We give an example to show that both the hypothesis \eqref{beta23cond} of Theorem \ref{generalPos} and the intersection condition of Corollary \ref{generalPosCor} are necessary.
\begin{example}
\label{gcdcounter}
Let $D_1,D_2,D_3$ be the coordinate lines in $\mathbb{P}^2$, let $p,q$ be rational primes, and let $S=\{p,q,\infty\}$, a set of places of $\mathbb{Q}$.  Then it follows from \cite[p.~707]{LevWirsing} that there exists a Zariski dense set of $(D_1+D_2+D_3,S)$-integral points $R$ in $\mathbb{P}^2(\mathbb{Q})$ such that
\begin{align*}
h_{D_i\cap D_j}(P)=\frac{1}{2}h(P)+O(1)
\end{align*}
for all $P\in R$ and all $i,j\in \{1,2,3\}$, $i\neq j$. Note that in this case, if $i\neq j$, then $D_i\cap D_j$ consists of a single point and $\beta(\CO(1),D_i\cap D_j)=\frac{2}{3}$.
\end{example}

\subsection{On the Diophantine Equation $f(a^m, y) = b^n$}\label{sec: gcdapp}

In this section, we provide an application of our result on greatest common divisors (Corollary \ref{generalPosCor}) to study the exponential Diophantine equation $f(a^m, y) = b^n$, where $f(x,y)$ is a polynomial with rational coefficients and $a$ and $b$ are positive integers with a nontrivial common factor.  Such an equation was studied by Corvaja and Zannier \cite{CZ00}, who noted that the equation did not seem to fall into prior treatments of Diophantine equations outside of very special situations (e.g., $f$ is homogeneous (Thue-Mahler), or more generally $f$ is homogeneous with respect to suitable weights). Corvaja and Zannier proved the following result:

\begin{theorem}[Corvaja-Zannier \cite{CZ00}]
\label{CZexp}
Let $f(x, y ) = a_0(x)y^d + a_1(x)y^{d-1}+\dots+a_d(X)$ be
a polynomial with rational coefficients, of degree $d \ge 2$ in $y$; let $a,b > 1$
be integers. Suppose that
\begin{enumerate}
    \item $a_0$ is constant,\label{CZi}
    \item the polynomial $f(0,y)$ has no repeated roots,\label{CZii}
    \item  $a$ and $b$ are not relatively prime.
\end{enumerate}
If the equation
$$ f(a^m, y) = b^n$$
has an infinite sequence of solutions $(m, n, y) \in \mathbb{Z} \times  \mathbb{Z} \times \mathbb{Z}$, such that
$\min\{|m|, |n|\} \rightarrow \infty$, then there exist an integer $h \neq 0$ and a polynomial
$p(x) \in \mathbb{Q}[x]$ such that $f(x^h, p(x))$ has only one term; furthermore, $a$ and  $b$ are multiplicatively dependent.
\end{theorem}

Using Corollary \ref{generalPosCor}, we derive a result along the lines of Theorem \ref{CZexp}. For purposes of comparison, we note that our result is weaker in some respects (condition \ref{CZi}' below is much stronger than condition \ref{CZi} in Theorem \ref{CZexp}), but under our hypotheses we partially weaken condition \ref{CZii} of Theorem \ref{CZexp}. Possibly, condition \ref{CZii} could also be weakened via the approach of Corvaja-Zannier in \cite{CZ00}; our main purpose here is to illustrate how our general Diophantine approximation results may be used to shed new light on existing problems.

\begin{theorem}
Let $F (X, Y, Z) \in \QQ[X, Y, Z]$ be a homogeneous polynomial of degree $d\ge 2$ and let $a,b > 1$ be integers. Suppose that 
\renewcommand{\labelenumi}{\arabic{enumi}'.}
\begin{enumerate}
\item $F(0,1,0) \neq 0$
\item Neither $F(0,y,z)$ nor $F(x,y,0)$ are powers of a linear form in $\Qbar[x,y,z]$.\label{tancond}
\item  $a$ and $b$ are not relatively prime.
\end{enumerate}
\renewcommand{\labelenumi}{\arabic{enumi}.}

Let $f(x,y)=F(x,y,1)$ Then the set of points 
\begin{align*}
\{(a^m, y) \in \ZZ \times \ZZ\mid f(a^m,y) = b^n, m,n\in\mathbb{Z}, m,n\geq 0\}, 
\end{align*}
is not Zariski dense in $\mathbb{A}^2$.
\end{theorem}

From the non-Zariski density statement one can derive a conclusion as in Theorem \ref{CZexp} (e.g., using the Lemma of \cite{CZ00} and Siegel's theorem on integral points on curves); we leave the details to the interested reader.

\begin{proof}

Let $V_i$ be the hypersurface of $\mathbb{P}^3$ defined by $b^iW^d=F(X,Y,Z)$, $i\in \{0,\ldots, d-1\}$. If $f(a^m,y)=b^n$, for some $m,n,y\in\mathbb{Z}$, then writing $b^n=(b^{n'})^db^i$ for some $n'\in\mathbb{Z}$ and $i\in \{0,\ldots, d-1\}$, we have $[W:X:Y:Z]=[b^{n'}:a^m:y:1]\in V_i(\mathbb{Q})$.  Then after fixing $i$ and letting $V=V_i$, it suffices to show that the set of points
\begin{align*}
R:=\{[b^n:a^m:y:1] \in V(\mathbb{Q})\mid m,n,y\in\mathbb{Z}, m,n\geq 0\}. 
\end{align*}
is not Zariski dense in $V$ (the statement in the theorem follows after taking an appropriate projection).

Let $D_1,D_2,D_3$ be the divisors on $V$ defined by $W=0$, $X=0$, and $Z=0$, respectively. Then clearly $D_1,D_2$, and $D_3$ are linearly equivalent ample effective Cartier divisors; let $D$ be any divisor in the same linear equivalence class. Since $F(0,1,0)\neq 0$, $V$ and the hyperplanes defined by $W=0$, $X=0$, and $Z=0$ are in general position on $\mathbb{P}^3$. Then by Remark \ref{remgenproper}, $D_1,D_2,D_3$ intersect properly on $V$. The condition \eqref{tancond} implies that $\Supp (D_1\cap D_2)(\Qbar)$ and $\Supp (D_1\cap D_3)(\Qbar)$ both contain more than one point, while the conditions $F(0,1,0)\neq 0$ and $d\geq 2$ imply that $\Supp (D_2\cap D_3)(\Qbar)$ contains more than one point. Moreover, if $S$ is the set of places of $\mathbb{Q}$ given by
\begin{align*}
S=\{p \text{ prime}\mid p|ab\}\cup \{\infty\},
\end{align*}
then $R$ is a set of $(D_1+D_2+D_3,S)$-integral points in $V(\mathbb{Q})$. By Corollary \ref{generalPosCor}, there exists a proper Zariski-closed subset $Z\subset V$ such that if $i\neq j$ and $v\in S$, then
\begin{align}
\min\{\lambda_{D_i,v}(P),\lambda_{D_j,v}(P)\}\leq \varepsilon h_D(P),\label{equgcd}
\end{align}
for all but finitely many $P\in R\setminus Z$.

Let $P=[b^n:a^m:y:1]\in R$. Let $h=h(P)=\log\max\{|b^n|,|a^m|,|y|\}=h_D(P)+O(1)$ and let $p$ be a prime dividing $\gcd(a,b)$. Then we compute
\begin{center}
\begin{tabular}{p{2cm}p{6cm}p{6cm}}
Divisor $D_i$ &  $\lambda_{D_i,\infty}(P)$  &  $\lambda_{D_i,p}(P)$ \\
$D_1$ &  $ h - n\log b$  &  $n (\log p) (\mathrm{ord}_p b)$ \\
$D_2$ & $h-m\log a$ & $m(\log p) (\mathrm{ord}_p a)$  \\
$D_3$ & $h$ & $0$\\
\end{tabular}
\end{center}

Then using \eqref{equgcd} with $(i,j,v)=(1,3,\infty), (2,3,\infty), (1,2,p)$, respectively, we find that for all $P\in R\setminus Z$,
\begin{align*}
n\geq \frac{1}{\log b}(1-\varepsilon)h_D(P)+O(1)\\
m\geq \frac{1}{\log a}(1-\varepsilon)h_D(P)+O(1)\\
\min\{m,n\}\leq 2\varepsilon h_D(P)+O(1).
\end{align*}
Taking $0<\varepsilon<\frac{1}{2\max\{\log a,\log b\}+1}$, the inequalities imply that $h_D(P)$ is bounded for $P\in R\setminus Z$. Since $D$ is ample, we conclude that $R\setminus Z$ is a finite set and $R$ is not Zariski dense in $V$.
\end{proof}

\subsection{Integral Points on the Complement of Three Numerically Parallel Curves Passing Through a Point}\label{sec: nongeneral}

From the work of Corvaja-Zannier \cite{CZ04} and the second author \cite{Lev09}, it is known that the complement of any $4$ ample divisors in general position on a projective surface does not contain a Zariski dense set of integral points; the number $4$ here is sharp as $\mathbb{G}_m^2\cong \mathbb{P}^2\setminus \{xyz=0\}$ and $\mathbb{G}_m^2(\CO_{K,S})\cong (\CO_{K,S}^*)^2$ is Zariski dense in $\mathbb{G}_m^2$ as long as $|S|>1$. As remarked after Theorem \ref{nonGeneralP2}, it is already an open problem to prove the degeneracy of integral points on the complement of three plane curves forming a normal crossings divisor of degree at least $4$. In contrast to this, it is sometimes possible to handle certain degenerate (i.e., non-normal crossings) configurations of three plane curves. For instance, the problem of integral points on the complement of a conic and two (distinct) tangent lines is easily reduced to Siegel's theorem for integral points on $\mathbb{P}^1$. A deeper result, depending ultimately on the Subspace Theorem, is the following theorem of Corvaja and Zannier \cite{CZ06}.

\begin{theorem}[Corvaja-Zannier]
\label{CZnongen}
Let $D_1, D_2, D_3$ be distinct, effective, irreducible, numerically equivalent divisors on a nonsingular projective surface $X$ defined over a number field $K$, such that
    \begin{enumerate}[(a)]
        \item $D_1 \cap D_2 \cap D_3$ consists of a single point, at which the $D_i$ intersect transversally.
        \item $D_i.D_j >1$ for some $i,j$.
    \end{enumerate}
Let $S$ be a finite set of places of $K$ containing the archimedean places.  Then no set of $(D_1 + D_2 + D_3, S)$-integral points in $X(K)$ is Zariski-dense in $X$.
\end{theorem}

We prove a generalization of Theorem \ref{CZnongen} where we greatly weaken the triple intersection condition, and we only require the three divisors to be numerically \emph{parallel} rather than numerically equivalent (to be precise, we don't recover the case $D_i.D_j=2$ of Theorem \ref{CZnongen}; however, in our formulation, which allows more than one point in $D_1\cap D_2\cap D_3$, excluding this case is necessary by Example \ref{12example}).

\begin{theorem}\label{nonGeneral}
Let $X$ be a projective surface over a number field $K$,  and let $D_1, D_2, D_3$ be effective Cartier divisors on $X$ pairwise intersecting properly. Suppose that there exist positive integers $a_1,a_2,a_3$ such that $a_1D_1, a_2D_2, a_3D_3$ are all numerically equivalent to an ample divisor $D$, and that
\begin{align*}
D_1\cap D_2\cap D_3\neq \emptyset.
\end{align*}
Let
\begin{align*}
\beta_0=\min_{\substack{i\neq j\\Q\in (D_1\cap D_2\cap D_3)(\kbar)}} \beta(D,(a_iD_i\cap a_jD_j)_Q).
\end{align*}
Furthermore, suppose that for every point $Q\in (D_1\cap D_2\cap D_3)(\kbar)$ and every permutation $i,j,k$ of the indices $1,2,3$, we have
\begin{align}
\label{beta1cond}
(\beta(D,(a_iD_i\cap a_jD_j)_Q)-1)+(\beta(D,(a_iD_i\cap a_kD_k)_Q)-1)(3\beta_0-2)>0. 
\end{align}
In particular, $\beta_0\geq \frac{2}{3}$ and \eqref{beta1cond} holds if
\begin{align}
\label{beta2cond}
\beta(D,(a_iD_i\cap a_jD_j)_Q)>1
\end{align}
or
\begin{align}
\label{49ineq2}
(D_i.D_j)_Q<\frac{4}{9}(D_i.D_j)
\end{align}
for all $Q\in (D_1\cap D_2\cap D_3)(\kbar)$ and all $i\neq j$ (where $(D_i.D_j)_Q$ denotes the local intersection multiplicity of $D_i$ and $D_j$ at $Q$, and $(D_i.D_j)$ the intersection multiplicity).  Let $S$ be a finite set of places of $K$ containing all the archimedean places. Then there exists a proper Zariski-closed subset $Z\subset X$ such that for any set $R\subset X(K)$ of $(D_1 + D_2 + D_3, S)$-integral points, the set $R\setminus Z$ is finite.
\end{theorem}

Theorem \ref{nonGeneralP2} from the introduction follows immediately  (using \eqref{49ineq2}).

\begin{proof}
After replacing $K$ by a finite extension, we can assume that every point in the support of $D_i\cap D_j$, $i\neq j$, is $K$-rational.

We first show that there exists a proper Zariski-closed subset $Z\subset X$ such that for any set $R\subset X(K)$ of $(D_1 + D_2 + D_3, S)$-integral points, we have
\begin{align*}
h_{a_1D_1\cap a_2D_2\cap a_3D_3}(P)&=\sum_{v \in S}\lambda_{a_1D_1\cap a_2D_2\cap a_3D_3, v}(P)+O(1)\\
&= \sum_{v \in S} \min\left(\lambda_{a_1D_1, v}(P), \lambda_{a_2D_2, v}(P),\lambda_{a_3D_3, v}(P)\right)+O(1)\\ 
&\le \varepsilon h_D(P)+O(1)
\end{align*}
for all $P\in R\setminus Z$.

%
%
 %

By definition and elementary properties of heights, for any $\varepsilon>0$,
\begin{align*}
\sum_{v\in S} \left(\lambda_{a_1D_{1},v}(P)+\lambda_{a_{2}D_{2},v}(P)+\lambda_{a_{3}D_{3},v}(P)\right)&=h_{a_1D_1}(P)+h_{a_2D_2}(P)+h_{a_3D_3}(P)+O(1)\\
&\geq (3-\varepsilon)h_D(P)+O(1),
\end{align*}
for all $P\in R$, where the $O(1)$ possibly depends on $R$ (but not $P$).

Let $P\in R$. For $v\in S$, let $\{i_v,j_v,k_v\}=\{1,2,3\}$ be such that

$$
\lambda_{a_{i_v}D_{i_v},v}(P)\geq \lambda_{a_{j_v}D_{j_v},v}(P)\geq \lambda_{a_{k_v}D_{k_v},v}(P).
$$

For each $v\in S$, there exists a point $Q_v\in \Supp (a_{i_v}D_{i_v}\cap a_{j_v}D_{j_v})$ (depending on $P$) such that
\begin{align*}
\lambda_{a_{i_v}D_{i_v}\cap a_{j_v}D_{j_v},v}(P)=\lambda_{(a_{i_v}D_{i_v}\cap a_{j_v}D_{j_v})_{Q_v},v}(P)+O(1),
\end{align*}
where the constant in the $O(1)$ is independent of $P$.

If $Q_v\not\in \Supp D_1\cap D_2\cap D_3$, then
\begin{align*}
\lambda_{a_1D_1\cap a_2D_2\cap a_3D_2,v}(P)&=\min \{\lambda_{a_{k_v}D_{k_v},v}(P),\lambda_{a_{i_v}D_{i_v}\cap a_{j_v}D_{j_v},v}(P)\}+O(1)\\
&=\min \{\lambda_{(a_{i_v}D_{i_v}\cap a_{j_v}D_{j_v})_{Q_v},v}(P),\lambda_{a_{k_v}D_{k_v},v}(P)\}+O(1)\\
&=\lambda_{(a_{i_v}D_{i_v}\cap a_{j_v}D_{j_v})_{Q_v}\cap (a_{k_v}D_{k_v}),v}\\
&=O(1)
\end{align*}
since $(a_{i_v}D_{i_v}\cap a_{j_v}D_{j_v})_{Q_v}\cap (a_{k_v}D_{k_v})$ is empty.  When $Q_v\in \Supp D_1\cap D_2\cap D_3$, we use the estimate
\begin{align*}
\lambda_{a_{k_v}D_{k_v},v}(P)=\lambda_{(a_1D_1\cap a_2D_2\cap a_3D_3)_{Q_v},v}(P)+O(1)\leq \lambda_{(a_{i_v}D_{i_v}\cap a_{k_v}D_{k_v})_{Q_v},v}(P)+O(1).
\end{align*}

Let 
\begin{align*}
S'&=\{v\in S\mid Q_v\not\in \Supp D_1\cap D_2\cap D_3\},\\
S''&=S\setminus S'.
\end{align*}

It follows that, up to $O(1)$,
\begin{align*}
(3-\varepsilon)h_D(P)&\leq \sum_{v\in S} \left(\lambda_{a_{i_v}D_{i_v},v}(P)+\lambda_{a_{j_v}D_{j_v},v}(P)+\lambda_{a_{k_v}D_{k_v},v}(P)\right)\\
&\leq \sum_{v\in S'} \left(\lambda_{a_{i_v}D_{i_v},v}(P)+\lambda_{(a_{i_v}D_{i_v}\cap a_{j_v}D_{j_v})_{Q_v},v}(P)\right)\\
&~~+\sum_{v\in S''} \left(\lambda_{a_{i_v}D_{i_v},v}(P)+\lambda_{(a_{i_v}D_{i_v}\cap a_{j_v}D_{j_v})_{Q_v},v}(P)+\lambda_{(a_{i_v}D_{i_v}\cap a_{k_v}D_{k_v})_{Q_v},v}(P)\right).
\end{align*}

Since $a_iD_i\equiv D$, we have $\beta(D,a_iD_i)=\frac{1}{3}$ for all $i$, and by Lemma \ref{betalcm}, for all $Q\in X(\overline{K})$ and $i\neq j$, we have
\begin{align*}
\beta(D,(a_iD_i\cap a_jD_j)_Q)\geq \beta(D,a_iD_i\cap a_jD_j)\geq \beta(D,a_iD_i)+\beta(D,a_jD_j)\geq \frac{2}{3}.
\end{align*}
Therefore $\beta_0\geq \frac{2}{3}$.

By Theorem \ref{BetaMain}, there exists a proper Zariski-closed subset $Z\subset X$ such that if $P\not\in Z$, 
\begin{align}
\sum_{v\in S} \lambda_{a_{i_v}D_{i_v},v}(P)+(3\beta(D,(a_{i_v}D_{i_v}\cap a_{j_v}D_{j_v})_{Q_v})-1)\lambda_{(a_{i_v}D_{i_v}\cap a_{j_v}D_{j_v})_{Q_v},v}(P)\leq (3+\varepsilon)h_D(P)\label{ijeq},
\end{align}
and
\begin{equation}
\begin{aligned}
&~~~\sum_{v\in S'} \lambda_{a_{i_v}D_{i_v},v}(P)+(3\beta(D,(a_{i_v}D_{i_v}\cap a_{j_v}D_{j_v})_{Q_v})-1)\lambda_{(a_{i_v}D_{i_v}\cap a_{j_v}D_{j_v})_{Q_v},v}(P)\\
&+\sum_{v\in S''} \lambda_{a_{i_v}D_{i_v},v}(P)+(3\beta(D,(a_{i_v}D_{i_v}\cap a_{k_v}D_{k_v})_{Q_v})-1)\lambda_{(a_{i_v}D_{i_v}\cap a_{k_v}D_{k_v})_{Q_v},v}(P)\\
&\leq (3+\varepsilon)h_D(P).\label{ikeq}
\end{aligned}
\end{equation}

Let $b = \frac{1}{3\beta_0-1}$. We have $0<b\leq 1$ and by definition of $\beta_0$, for any $i,j\in\{1,2,3\}, i\neq j$, and $Q\in (D_1\cap D_2\cap D_3)(\kbar)$, 
\begin{align*}
b\bigg( 3\beta(D,(a_{i}D_{i}\cap a_{j}D_{j})_{Q})-1 \bigg) \geq 1. 
\end{align*}
Note also that 
\begin{align*}
\lambda_{(a_{i_v}D_{i_v}\cap a_{j_v}D_{j_v})_{Q_v},v}(P)\geq \lambda_{(a_{i_v}D_{i_v}\cap a_{k_v}D_{k_v})_{Q_v},v}(P)+O(1).
\end{align*}

Let 
\begin{align*}
\gamma=3\min_{\substack{\{i,j,k\}=\{1,2,3\}\\Q\in (D_1\cap D_2\cap D_3)(\kbar)}}\frac{(\beta(D,(a_iD_i\cap a_jD_j)_Q)-1)+(\beta(D,(a_iD_i\cap a_kD_k)_Q)-1)(3\beta_0-2)}{3\beta_0-1}.
\end{align*}
By hypothesis, $\gamma>0$. Fixing $Q\in (D_1\cap D_2\cap D_3)(\kbar)$, $P$, and $v$, for $i,j\in\{1,2,3\}$ let 
\begin{align*}
\beta_{ij}&=\beta(D,(a_{i}D_{i}\cap a_{j}D_{j})_{Q}),\\
\lambda_{ij}&=\lambda_{(a_{i}D_{i}\cap a_{j}D_{j})_{Q},v}(P). 
\end{align*}
Then assuming $\lambda_{ij}\geq \lambda_{ik}$, we compute
\begin{align*}
b(3\beta_{ij}-1)\lambda_{ij}+(1-b)(3\beta_{ik}-1)\lambda_{ik}&\geq \lambda_{ij}+(b(3\beta_{ij}-1)-1+(1-b)(3\beta_{ik}-1))\lambda_{ik}\\
&\geq \lambda_{ij}+\left(1+\frac{3((\beta_{ij}-1)+(\beta_{ik}-1)(3\beta_0-2))}{3\beta_0-1}\right)\lambda_{ik}\\
&\geq \lambda_{ij}+(1+\gamma)\lambda_{ik}.
\end{align*}

Multiplying \eqref{ijeq} by $b$, \eqref{ikeq} by $1-b$, and summing and using the above calculation, we obtain
\begin{equation*}
\begin{aligned}
&~~~\sum_{v\in S'} \lambda_{a_{i_v}D_{i_v},v}(P)+(3\beta(D,(a_{i_v}D_{i_v}\cap a_{j_v}D_{j_v})_{Q_v})-1)\lambda_{(a_{i_v}D_{i_v}\cap a_{j_v}D_{j_v})_{Q_v},v}(P)\\
&+\sum_{v\in S''} \lambda_{a_{i_v}D_{i_v},v}(P)+\lambda_{(a_{i_v}D_{i_v}\cap a_{j_v}D_{j_v})_{Q_v},v}(P)+(1+\gamma)\lambda_{(a_{i_v}D_{i_v}\cap a_{k_v}D_{k_v})_{Q_v},v}(P)\\
&\leq (3+\varepsilon)h_D(P)+O(1).
\end{aligned}
\end{equation*}
Then there exists a proper Zariski-closed subset $Z\subset X$ such that if $P\not\in Z$, up to $O(1)$,
\begin{equation*}
\begin{aligned}
&~~~(3-\varepsilon)h_D(P)+\gamma\sum_{v\in S} \lambda_{a_1D_1\cap a_2D_2\cap a_3D_3,v}(P)\\
&\leq (3-\varepsilon)h_D(P)+\gamma\sum_{v\in S''} \lambda_{(a_{i_v}D_{i_v}\cap a_{k_v}D_{k_v})_{Q_v},v}(P)\\
&\leq \sum_{v\in S'} \left(\lambda_{a_{i_v}D_{i_v},v}(P)+\lambda_{(a_{i_v}D_{i_v}\cap a_{j_v}D_{j_v})_{Q_v},v}(P) \right)\\
&~~~+\sum_{v\in S''} \left(\lambda_{a_{i_v}D_{i_v},v}(P)+\lambda_{(a_{i_v}D_{i_v}\cap a_{j_v}D_{j_v})_{Q_v},v}(P)+(1+\gamma)\lambda_{(a_{i_v}D_{i_v}\cap a_{k_v}D_{k_v})_{Q_v},v}(P) \right)\\
&\leq (3+\varepsilon)h_D(P).
\end{aligned}
\end{equation*}

Therefore, if $P\in R\setminus Z$,
\begin{align*}
\sum_{v\in S}\lambda_{a_1D_1\cap a_2D_2\cap a_3D_3,v}(P)<\frac{2\varepsilon}{\gamma}h_D(P)+O(1).
\end{align*}
Equivalently, since $R$ is a set of $(D_1+D_2+D_3,S)$-integral points and $\Supp (a_1D_1\cap a_2D_2\cap a_3D_3)\subset \Supp (D_1+D_2+D_3)$, for given $\varepsilon>0$, there exists a proper Zariski-closed subset $Z\subset X$ such that for $P\in R\setminus Z$,
\begin{align}
\label{tripineq}
h_{a_1D_1\cap a_2D_2\cap a_3D_3}(P)<\varepsilon h_D(P)+O(1),
\end{align}
finishing the proof of the claim.

Let $Q$ be some point in $\Supp D_1\cap D_2\cap D_3$ (which is nonempty by assumption).   Let $\pi:\tilde{X}\to X$ be the blowup at $Q$, with exceptional divisor $E$. If $R$ is a set of $(D_1+D_2+D_3,S)$-integral points in $X(K)$, then $\pi^{-1}(R)\setminus E$ is a set of $(\pi^*D_1+\pi^*D_2+\pi^*D_3,S)$-integral points in $\tilde{X}(K)$. So it suffices to show that there exists a proper Zariski-closed subset $\tilde{Z}$ of $\tilde{X}$ such that for any set $\tilde{R}$ of $(\pi^*D_1+\pi^*D_2+\pi^*D_3,S)$-integral points in $\tilde{X}(K)$, the set $\tilde{R}\setminus \tilde{Z}$ is finite.  Define the effective Cartier divisors 
\begin{align*}
D_i'=a_i\pi^*D_i-E, \quad i=1,2,3.
\end{align*}

Let $\tilde{R}$ be a set of $(\pi^*D_1+\pi^*D_2+\pi^*D_3,S)$-integral points in $\tilde{X}(K)$ (and hence a set of $(D_1'+D_2'+D_3',S)$-integral points). For $P\in \tilde{R}$ and $\varepsilon>0$, we have
\begin{align}
\label{strictheighteq}
\sum_{v\in S} \left(\lambda_{D_1',v}(P)+\lambda_{D_2',v}(P)+\lambda_{D_3',v}(P)\right)&=h_{D_1'}(P)+h_{D_2'}(P)+h_{D_3'}(P)+O(1)\\
&\geq (3-\varepsilon)h_{\pi^*D}(P)-(3-\varepsilon)h_E(P)+O(1)\notag,
\end{align}
where the $O(1)$ possibly depends on $\tilde{R}$ (but not $P$). 

We now bound the left-hand side of the above equation. As in previous arguments, it suffices to bound a sum of the form
\begin{align*}
\sum_{v\in S} \left(\lambda_{D_{i_v}',v}(P)+\lambda_{D_{i_v}'\cap D_{j_v}',v}(P)+\lambda_{D_{i_v}'\cap D_{j_v}'\cap D_{k_v}',v}(P)\right),
\end{align*}
where $\{i_v,j_v,k_v\}=\{1,2,3\}$ for $v\in S$. We first note that it follows from \eqref{tripineq} and functoriality that given $\varepsilon>0$, there exists  a proper Zariski-closed subset $\tilde{Z}\subset \tilde{X}$ such that
\begin{align*}
\sum_{v\in S} \lambda_{D_{i_v}'\cap D_{j_v}'\cap D_{k_v}',v}(P)<\varepsilon h_{\pi^*D}(P)+O(1)
\end{align*}
for all $P\in \tilde{R}\setminus \tilde{Z}$. For the same reasons, we may choose $\tilde{Z}$ so that we also have
\begin{align}
\label{EpiDineq}
h_E(P)<\varepsilon h_{\pi^*D}(P)+O(1)
\end{align}
for all $P\in \tilde{R}\setminus \tilde{Z}$. We can write (as closed subschemes)
\begin{align*}
D_{i_v}'\cap D_{j_v}'=Y_{0,v}+Y_{1,v},
\end{align*}
where $\Supp \pi(Y_{1,v})=Q$, $\dim Y_{0,v}=0$, and $Y_{0,v}\cap E=\emptyset$. We have (for an appropriate $\tilde{Z}$)
\begin{align*}
\sum_{v\in S} \lambda_{Y_{1,v},v}(P)<\varepsilon h_{\pi^*D}(P)+O(1)
\end{align*}
for all $P\in \tilde{R}\setminus \tilde{Z}$, and so
\begin{align*}
\lambda_{D_{i_v}'\cap D_{j_v}',v}(P)\leq \lambda_{Y_{0,v},v}(P)+\varepsilon h_{\pi^*D}(P)+O(1)
\end{align*}
for all $P\in \tilde{R}\setminus \tilde{Z}$.

Let $\delta\in\mathbb{Q}, \delta>0$, be chosen as in Lemma \ref{exclemma} so that \eqref{exceqn} holds, and let 
\begin{align*}
\gamma'=\beta(\pi^*D-\delta E, \pi^*D-E)-\frac{1}{3}>0.
\end{align*}

Note that
\begin{align*}
\beta(\pi^*D-\delta E, Y_{0,v})\geq \beta(\pi^*D-\delta E, D_{i,v}')+\beta(\pi^*D-\delta E, D_{j,v}'). 
\end{align*}

This does not follow directly from Lemma \ref{betalcm} (since $D_{i,v}'$ and $D_{j,v}'$ may not intersect properly above $Q$, along the component $Y_{1,v}$), but it follows from a slight modification to the proof of that lemma as $D_i'$ and $D_j'$ intersect properly in a neighborhood of the zero-dimensional closed subscheme $Y_{0,v}$.

Using Theorem \ref{BetaMain}, for any $\varepsilon>0$ we find that for $P\in \tilde{R}$ outside a proper Zariski-closed subset of $\tilde{X}$ (and up to $O(1)$)
\begin{align*}
&~~~\left(\frac{1}{3}+\gamma'\right)\sum_{v\in S} \left(\lambda_{D_{i_v}',v}(P)+\lambda_{D_{i_v}'\cap D_{j_v}',v}(P)+\lambda_{D_{i_v}'\cap D_{j_v}'\cap D_{k_v}',v}(P)\right) \\
&\leq\sum_{v\in S} \left(\beta(\pi^*D-\delta E, D_{i,v}')\lambda_{D_{i_v}',v}(P)+(\beta(\pi^*D-\delta E, Y_{0,v})-\beta(\pi^*D-\delta E, D_{i,v}'))\lambda_{Y_{0,v},v}(P)\right)\\
&~~~+\varepsilon h_{\pi^*D}(P)\\
&\leq (1+\varepsilon)h_{\pi^*D-\delta E}(P)+\varepsilon h_{\pi^*D}(P)\\
&\leq (1+2\varepsilon) h_{\pi^*D}(P).
\end{align*}
Therefore, for some positive $\delta'>0$, for $P\in \tilde{R}$ outside a proper Zariski-closed subset of $\tilde{X}$ we have
\begin{align*}
\sum_{v\in S} \left(\lambda_{D_1',v}(P)+\lambda_{D_2',v}(P)+\lambda_{D_3',v}(P)\right)&\leq (3-\delta') h_{\pi^*D}(P).
\end{align*}

On the other hand, by \eqref{strictheighteq} and \eqref{EpiDineq} (taking $\varepsilon$ sufficiently small), for $P\in \tilde{R}$ outside a proper Zariski-closed subset of $\tilde{X}$,
\begin{align*}
\sum_{v\in S} \left(\lambda_{D_1',v}(P)+\lambda_{D_2',v}(P)+\lambda_{D_3',v}(P)\right)&\geq \left(3 - \frac{\delta'}{2} \right) h_{\pi^*D}(P).
\end{align*}

Finally, since $\pi^*D$ is big, combining the above inequalities with an application of Northcott's theorem (for big divisors) gives that there exists a proper Zariski-closed subset $\tilde{Z}$ of $\tilde{X}$ such that $\tilde{R}\setminus \tilde{Z}$ is finite.
\end{proof}

We now give two examples addressing the sharpness of conditions \eqref{beta2cond} and \eqref{49ineq2}. Both examples are based on the following construction:
\begin{example}
\label{pencil}
Let $K$ be a number field and let $S$ be a finite set of places of $K$ containing the archimedean places with $|S|\geq 2$. Let $\Lambda$ be a linear pencil of curves in a projective surface $X$ such that the general member of $\Lambda$ is isomorphic to $\mathbb{P}^1$, and the pencil $\Lambda$ has exactly two base points $Q_1,Q_2$, which, after possibly enlarging $K$, we may assume are $K$-rational (our construction would also work if there is a single base point). Let $D_1,D_2,D_3\in \Lambda$ be distinct elements. Then we claim that the conclusion of Theorem \ref{nonGeneral} does not hold, i.e., there does not exist a proper Zariski-closed subset $Z\subset X$ such that for any set $R\subset X(K)$ of $(D_1 + D_2 + D_3, S)$-integral points, the set $R\setminus Z$ is finite.   Let $C\in \Lambda \setminus\{D_1,D_2,D_3\}$ be a general member. We claim that $C$ contains an infinite set of $(D_1+D_2+D_3,S)$-integral points. Indeed, any two elements of $\Lambda$ intersect only at the points $Q_1$ and $Q_2$, and therefore $C\cap (D_1\cup D_2\cup D_3)=\{Q_1,Q_2\}$. Since $C\cong \mathbb{P}^1$, we have $C\setminus (D_1\cup D_2\cup D_3)\cong \mathbb{G}_m$, and $C$ will contain an infinite set $R$ of $(D_1+D_2+D_3,S)$-integral points (we use $|S|\geq 2$ here, so that $\CO_{K,S}^*$ is infinite). Since the union of such elements $C\in \Lambda$ is Zariski dense in $X$, the conclusion of Theorem \ref{nonGeneral} does not hold. \footnote{We do not assert that there exists a Zariski dense set of  $(D_1 + D_2 + D_3, S)$-integral points; in fact, the pencil yields a morphism $X\setminus \{D_1\cup D_2\cup D_3\}\to\mathbb{P}^1\setminus \{0,1,\infty\}$ and Siegel's theorem shows that there is no Zariski dense set of $(D_1 + D_2 + D_3, S)$-integral points in $X(K)$.} 
\end{example}

We first give an example showing that the factor $\frac{4}{9}$ in the intersection condition \eqref{49ineq} (or \eqref{49ineq2}) cannot be replaced by anything larger than $\frac{1}{2}$.

\begin{example}\label{12example}
Consider the pencil of plane conics $\Lambda=\{C_\lambda\mid \lambda\in K\}$, where  $C_\lambda  = \{y^2 - \lambda xz = 0\}$. Then any two distinct curves $C_\lambda, C_\lambda'\in \Lambda$ intersect precisely at the two points $Q_1=[0:0:1], Q_2=[1:0:0]$, each with multiplicity $2$. By Example \ref{pencil}, if $D_1,D_2,D_3\in \Lambda$ are distinct elements, then the conclusion of Theorem \ref{nonGeneral} does not hold (for appropriate $K$ and $S$), and we  note that
\begin{align*}
(D_i.D_j)_{Q_1}=(D_i.D_j)_{Q_2}=2=\frac{1}{2}(D_i.D_j).
\end{align*}

\end{example}

Somewhat surprisingly, the next example shows that the beta constant condition in Theorem \ref{nonGeneral} is sharp, in the sense that the condition
\begin{align*}
\beta(D,(a_iD_i\cap a_jD_j)_Q)>1
\end{align*}
cannot be replaced by the inequality
\begin{align*}
\beta(D,(a_iD_i\cap a_jD_j)_Q)\geq 1.
\end{align*}

\begin{example}\label{beta1example}
Let $D_1$ and $D_2$ be two distinct irreducible curves of type $(1,1)$ on $\mathbb{P}^1\times \mathbb{P}^1$, defined over a number field $K$, intersecting in two distinct $K$-rational points $Q_1$ and $Q_2$. Let $\Lambda$ be the pencil of curves containing $D_1$ and $D_2$, and let $D_3\in \Lambda\setminus\{D_1,D_2\}$ be another irreducible element of the pencil.  

For any point $Q\in (\mathbb{P}^1\times \mathbb{P}^1)(K)$ (viewed as a closed subscheme with the reduced induced structure) an elementary computation gives
\begin{align*}
\beta(D_1,Q)=1.
\end{align*}
Since $D_i\cap D_j=Q_1+Q_2$ (as closed subschemes) when $i\neq j$, we see from Example \ref{pencil} (with this pencil) that the conclusion of Theorem \ref{nonGeneral} does not hold in general when \eqref{beta1cond} is replaced by $\beta(D,(a_iD_i\cap a_jD_j)_Q)\geq 1$.

\end{example}

\subsection{Families of Unit Equations}

The unit equation theorem, proved by Siegel (when $S$ is the set of archimedean places) and Mahler, is a fundamental and ubiquitous result in number theory:
\begin{theorem}[Siegel-Mahler]
    Let $K$ be a number field and let $S$ be a finite set of places of $K$ containing the archimedean places. Let $\cO_{K,S}$ be the ring of $S$-integers of $K$ and let $\cO_{K,S}^*$ be the group of $S$-units of $K$. Let $a,b,c \in K^*$. The $S$-unit equation
    \begin{align*}
        au+bv=c, \quad u,v \in \cO_{K,S}^*,
    \end{align*}
has only finitely many solutions. 
\end{theorem}

From another viewpoint, the theorem is equivalent to Siegel's theorem on integral points on the affine curve $\bP^1 \setminus \{0,1,\infty\}$.  In this section, we study the one-parameter family of unit equations
\begin{align*}
f_1(t)u+f_2(t)v=f_3(t), \quad t\in \CO_{K,S}, u,v\in \CO_{K,S}^*,
\end{align*}
where $f_1,f_2,f_3$ are polynomials over a number field $K$. This equation was treated in the case $\deg f_1=\deg f_2=\deg f_3$ by Corvaja and Zannier \cite{CZ06, CZ10}, and in the case $\deg f_1+\deg f_2=\deg f_3$ by the second author \cite{Lev06}. Applying the results of the previous section to certain surfaces, we handle a wide range of new values of the triple $(d_1,d_2,d_3)=(\deg f_1,\deg f_2,\deg f_3)$.

\begin{theorem}
\label{thfamily}
Let $f_1,f_2,f_3\in K[t]$ be nonconstant polynomials without a common zero of degrees $d_1,d_2,d_3$, respectively.  Let $\{i_1,i_2,i_3\}=\{1,2,3\}$ be such that $d_{i_1}\geq d_{i_2}\geq d_{i_3}$ and suppose that
\begin{align}
\label{familyineq}
\left(\sqrt{\frac{d_{i_1}+1}{d_{i_1}+1-d_{i_2}}}-1\right)\left(\sqrt{\frac{{d_{i_1}+1}}{d_{i_1}+1-d_{i_3}}}-1\right)>\frac{1}{4}.
\end{align}
In particular, \eqref{familyineq} holds when
\begin{align*}
\max_i (d_i+1)<\frac{9}{5}\min_i d_i.
\end{align*}
Then the set of solutions $(t,u,v)\in\mathbb{A}^3(K)$ of the equation
\begin{align}
\label{funiteq}
f_1(t)u+f_2(t)v=f_3(t), \quad t\in \CO_{K,S}, u,v\in \CO_{K,S}^*,
\end{align}
is contained in a finite number of rational curves in $\mathbb{A}^3$.
\end{theorem}

We note that if $d_{i_1}>d_{i_2}$ in Theorem \ref{thfamily}, then by \cite[Lemma~5]{Lev06}, one may drop the integrality condition on $t$ in \eqref{funiteq} (i.e., consider solutions with $t\in K$).

\begin{proof}
We first remark that the statement of the theorem is independent of the ordering of $f_1,f_2,f_3$; this follows from noting, for example, that if $(t_0,u_0,v_0)\in \CO_{K,S}\times \CO_{K,S}^*\times \CO_{K,S}^*$ is a solution to $f_1(t)u+f_3(t)v=f_2(t)$, then $(t_0,-u_0/v_0,1/v_0)\in \CO_{K,S}\times \CO_{K,S}^*\times \CO_{K,S}^*$ is a solution to $f_1(t)u+f_2(t)v=f_3(t)$, and this relation corresponds to a birational automorphism of $\mathbb{A}^3$. Thus, after permuting the $f_i$, we may assume that $d_1\geq d_2\geq d_3>0$, $d_{i_j}=d_j$, $j=1,2,3$, and 
\begin{align*}
\left(\sqrt{\frac{d_1+1}{d_{1}+1-d_{2}}}-1\right)\left(\sqrt{\frac{{d_1+1}}{d_{1}+1-d_{3}}}-1\right)>\frac{1}{4}.
\end{align*}
In place of \eqref{funiteq}, we will actually study the slightly modified equation
\begin{align}
\label{funiteq2}
f_1(t)u+f_2(t)v^{d_1+1-d_2}=f_3(t), \quad t\in \CO_{K,S}, u,v\in \CO_{K,S}^*.
\end{align}
Since $\CO_{K,S}^*$ is finitely generated, we can find a number field $L$ and a finite set of places $S'$ of $L$ such that every element of $\CO_{K,S}^*$ has a $(d_1+1-d_2)$-th root in $\CO_{L,S'}^*$. Then \eqref{funiteq} reduces to studying the equation \eqref{funiteq2} (with $(K,S)$ replaced by $(L,S')$).

We let $F_i\in K[x_0,x_1,x_2,x_3]$ be the homogeneous polynomial $F_i=f_i(x_0/x_3)x_3^{d_i}$, $i=1,2,3$, and let $X$ be the hypersurface in $\mathbb{P}^3$ defined by the equation
\begin{align*}
F:=x_1F_1+x_2^{d_1+1-d_2}F_2-x_3^{d_1+1-d_3}F_3=0.
\end{align*}

Since the polynomials $f_i$ do not have a common zero, it follows easily that $F$ is irreducible in $\kbar[x_0,x_1,x_2,x_3]$ (and $X$ is a projective surface).

Let $H_i$ be the hyperplane of $\mathbb{P}^3$ defined by $x_i=0$, $i=1,2,3$, and let $D_i=H_i|_X$ be the divisor on $X$ defined by $x_i=0$, $i=1,2,3$. Since $X$ is a hypersurface in $\mathbb{P}^3$ and $H_i,H_j,X$, $i\neq j$, are in general position on $\mathbb{P}^3$ (as $F_i(1,0,0,0)\neq 0$), by Remark \ref{remgenproper}, $D_i$ and $D_j$ intersect properly on $X$ if $i\neq j$. Letting $P_0=[1:0:0:0]\in X(K)$, we have $D_1\cap D_2\cap D_3=\{P_0\}\neq \emptyset$, and by construction $D_1\sim D_2 \sim D_3$.  We also set $D=D_1$. 

We note that $P_0$ is a nonsingular point of $X$, with maximal ideal $(x_2,x_3)$ in the local ring at $P_0$. Then if $\{i,j,k\}=\{1,2,3\}$, from the equation for $F$ we find $(x_i,x_j)\CO_{P_0}=(x_i,x_j,x_k^{d_1+1-d_k})\CO_{P_0}$. It follows that
\begin{align*}
(D_i.D_j)_{P_0}=d_1+1-d_k.
\end{align*}
Note also that $D^2=\deg X=d_1+1$. Then by Lemma \ref{betamultsimple},
\begin{align*}
\beta(D,(D_i\cap D_j)_{P_0})&\geq \frac{2}{3}\sqrt{\frac{D^2}{(D_i.D_j)_{P_0}}}\\
&\geq  \frac{2}{3}\sqrt{\frac{d_1+1}{d_1+1-d_k}}.
\end{align*}

Then it is not hard to show that the condition \eqref{beta1cond} of Theorem \ref{nonGeneral} is satisfied if
\begin{align*}
\left(\frac{2}{3}\sqrt{\frac{d_1+1}{d_1+1-d_3}}-1\right)+\left(\frac{2}{3}\sqrt{\frac{d_1+1}{d_1+1-d_2}}-1\right)\left(2\sqrt{\frac{d_1+1}{d_1+1-d_3}}-2\right)>0,
\end{align*}
which is equivalent to \eqref{familyineq}. Now we let
\begin{align*}
R=\{[t:u:v:1]\in X(K)\mid f_1(t)u+f_2(t)v^{d_1+1-d_2}=f_3(t), (t,u,v)\in \CO_{K,S}\times \CO_{K,S}^* \times \CO_{K,S}^*\}.  
\end{align*}
Then $R$ is easily seen to be a set of $(D_1+D_2+D_3,S)$-integral points on $X$. Now under the assumption \eqref{familyineq}, we can apply Theorem \ref{nonGeneral} with the divisors $D_1,D_2,D_3$ on $X$ and conclude that $R$ lies in a finite union of curves on $X\subset\mathbb{P}^3$. Finally, we note that since the divisors $D_i$, $i=1,2,3$ are ample on $X$, Siegel's theorem on integral points on curves implies that if $C$ is a curve in $X$ and $C\cap R$ is infinite, then $C$ is a rational curve, and thus we may take a finite union of rational curves in the conclusion of the theorem.
\end{proof}

\section*{Acknowledgment}
We would like to thank Yizhen Zhao for helpful discussions. 

\bibliographystyle{amsalpha}
\bibliography{refFile}{}

\providecommand{\bysame}{\leavevmode\hbox to3em{\hrulefill}\thinspace}
\providecommand{\MR}{\relax\ifhmode\unskip\space\fi MR }
\providecommand{\MRhref}[2]{%
  \href{http://www.ams.org/mathscinet-getitem?mr=#1}{#2}
}
\providecommand{\href}[2]{#2}
\begin{thebibliography}{BCZ03}

\bibitem[Aut09]{Aut09}
Pascal Autissier, \emph{G\'{e}om\'{e}tries, points entiers et courbes
  enti\`eres}, Ann. Sci. \'{E}c. Norm. Sup\'{e}r. (4) \textbf{42} (2009),
  no.~2, 221--239. \MR{2518077}

\bibitem[Aut11]{Aut11}
\bysame, \emph{Sur la non-densit\'{e} des points entiers}, Duke Math. J.
  \textbf{158} (2011), no.~1, 13--27. \MR{2794367}

\bibitem[BCZ03]{BCZ}
Yann Bugeaud, Pietro Corvaja, and Umberto Zannier, \emph{An upper bound for the
  {G}.{C}.{D}. of {$a^n-1$} and {$b^n-1$}}, Math. Z. \textbf{243} (2003),
  no.~1, 79--84. \MR{1953049}

\bibitem[CEL01]{CEL}
Steven~Dale Cutkosky, Lawrence Ein, and Robert Lazarsfeld, \emph{Positivity and
  complexity of ideal sheaves}, Math. Ann. \textbf{321} (2001), no.~2,
  213--234. \MR{1866486}

\bibitem[CZ00]{CZ00}
Pietro Corvaja and Umberto Zannier, \emph{On the {D}iophantine equation
  {$f(a^m,y)=b^n$}}, Acta Arith. \textbf{94} (2000), no.~1, 25--40.
  \MR{1762454}

\bibitem[CZ04a]{CZ04'}
\bysame, \emph{On a general {T}hue's equation}, Amer. J. Math. \textbf{126}
  (2004), no.~5, 1033--1055. \MR{2089081}

\bibitem[CZ04b]{CZ04}
\bysame, \emph{On integral points on surfaces}, Ann. of Math. (2) \textbf{160}
  (2004), no.~2, 705--726. \MR{2123936}

\bibitem[CZ05]{CZ05}
\bysame, \emph{A lower bound for the height of a rational function at
  {$S$}-unit points}, Monatsh. Math. \textbf{144} (2005), no.~3, 203--224.
  \MR{2130274}

\bibitem[CZ06]{CZ06}
\bysame, \emph{On the integral points on certain surfaces}, Int. Math. Res.
  Not. (2006), Art. ID 98623, 20. \MR{2219222}

\bibitem[CZ10]{CZ10}
\bysame, \emph{Integral points, divisibility between values of polynomials and
  entire curves on surfaces}, Adv. Math. \textbf{225} (2010), no.~2,
  1095--1118. \MR{2671189}

\bibitem[EF08]{EF08}
Jan-Hendrik Evertse and Roberto~G. Ferretti, \emph{A generalization of the
  {S}ubspace {T}heorem with polynomials of higher degree}, Diophantine
  approximation, Dev. Math., vol.~16, SpringerWienNewYork, Vienna, 2008,
  pp.~175--198. \MR{2487693}

\bibitem[Eis95]{Eis95}
David Eisenbud, \emph{Commutative algebra}, Graduate Texts in Mathematics, vol.
  150, Springer-Verlag, New York, 1995, With a view toward algebraic geometry.
  \MR{1322960}

\bibitem[Ful89]{Fulton}
William Fulton, \emph{Algebraic curves}, Advanced Book Classics, Addison-Wesley
  Publishing Company, Advanced Book Program, Redwood City, CA, 1989, An
  introduction to algebraic geometry, Notes written with the collaboration of
  Richard Weiss, Reprint of 1969 original. \MR{1042981}

\bibitem[GW20]{UW}
Ulrich G\"{o}rtz and Torsten Wedhorn, \emph{Algebraic geometry {I}.
  {S}chemes---with examples and exercises}, second ed., Springer Studium
  Mathematik---Master, Springer Spektrum, Wiesbaden, [2020] \copyright 2020.
  \MR{4225278}

\bibitem[Har77]{Harts}
Robin Hartshorne, \emph{Algebraic geometry}, Graduate Texts in Mathematics, No.
  52, Springer-Verlag, New York-Heidelberg, 1977. \MR{0463157}

\bibitem[HL21]{HL21}
Gordon Heier and Aaron Levin, \emph{A generalized {S}chmidt subspace theorem
  for closed subschemes}, Amer. J. Math. \textbf{143} (2021), no.~1, 213--226.
  \MR{4201783}

\bibitem[HL22]{HL22}
Keping Huang and Aaron Levin, \emph{Greatest common divisors on the complement
  of numerically parallel divisors}, 2022.

\bibitem[Laz04]{Laz04}
Robert Lazarsfeld, \emph{Positivity in algebraic geometry. {I}}, Ergebnisse der
  Mathematik und ihrer Grenzgebiete. 3. Folge. A Series of Modern Surveys in
  Mathematics [Results in Mathematics and Related Areas. 3rd Series. A Series
  of Modern Surveys in Mathematics], vol.~48, Springer-Verlag, Berlin, 2004,
  Classical setting: line bundles and linear series. \MR{2095471}

\bibitem[Lev06]{Lev06}
Aaron Levin, \emph{One-parameter families of unit equations}, Math. Res. Lett.
  \textbf{13} (2006), no.~5-6, 935--945. \MR{2280786}

\bibitem[Lev09]{Lev09}
\bysame, \emph{Generalizations of {S}iegel's and {P}icard's theorems}, Ann. of
  Math. (2) \textbf{170} (2009), no.~2, 609--655. \MR{2552103}

\bibitem[Lev14]{LevWirsing}
\bysame, \emph{Wirsing-type inequalities}, Bull. Inst. Math. Acad. Sin. (N.S.)
  \textbf{9} (2014), no.~4, 685--710. \MR{3309948}

\bibitem[Lev19]{Lev19}
\bysame, \emph{Greatest common divisors and {V}ojta's conjecture for blowups of
  algebraic tori}, Invent. Math. \textbf{215} (2019), no.~2, 493--533.
  \MR{3910069}

\bibitem[Mat80]{Matsumura}
Hideyuki Matsumura, \emph{Commutative algebra}, second ed., Mathematics Lecture
  Note Series, vol.~56, Benjamin/Cummings Publishing Co., Inc., Reading, Mass.,
  1980. \MR{575344}

\bibitem[MR15]{MR15}
David McKinnon and Mike Roth, \emph{Seshadri constants, diophantine
  approximation, and {R}oth's theorem for arbitrary varieties}, Invent. Math.
  \textbf{200} (2015), no.~2, 513--583. \MR{3338009}

\bibitem[RV20]{RV20}
Min Ru and Paul Vojta, \emph{A birational {N}evanlinna constant and its
  consequences}, Amer. J. Math. \textbf{142} (2020), no.~3, 957--991.
  \MR{4101336}

\bibitem[RW22]{RW22}
Min Ru and Julie Tzu-Yueh Wang, \emph{The {R}u-{V}ojta result for
  subvarieties}, Int. J. Number Theory \textbf{18} (2022), no.~1, 61--74.
  \MR{4369792}

\bibitem[Sch77]{Schl77}
H.~P. Schlickewei, \emph{The {${\mathfrak p}$}-adic
  {T}hue-{S}iegel-{R}oth-{S}chmidt theorem}, Arch. Math. (Basel) \textbf{29}
  (1977), no.~3, 267--270.

\bibitem[Sil87]{Sil87}
Joseph~H. Silverman, \emph{Arithmetic distance functions and height functions
  in {D}iophantine geometry}, Math. Ann. \textbf{279} (1987), no.~2, 193--216.
  \MR{919501}

\bibitem[Sil05]{Sil05}
\bysame, \emph{Generalized greatest common divisors, divisibility sequences,
  and {V}ojta's conjecture for blowups}, Monatsh. Math. \textbf{145} (2005),
  no.~4, 333--350. \MR{2162351}

\bibitem[Voj87]{Voj87}
Paul Vojta, \emph{Diophantine approximations and value distribution theory},
  Lecture Notes in Mathematics, vol. 1239, Springer-Verlag, Berlin, 1987.

\bibitem[Voj97]{Voj97}
\bysame, \emph{On {C}artan's theorem and {C}artan's conjecture}, Amer. J. Math.
  \textbf{119} (1997), no.~1, 1--17.

\bibitem[Voj20]{Voj20}
\bysame, \emph{Birational nevanlinna constants, beta constants, and diophantine
  approximation to closed subschemes}, arXiv preprint arXiv:2008.00405 (2020).

\bibitem[Voj23]{Voj23}
\bysame, \emph{Birational {N}evanlinna constants, beta constants, and
  diophantine approximation to closed subschemes}, J. Th\'{e}or. Nombres
  Bordeaux \textbf{35} (2023), no.~1, 17--61. \MR{4596522}

\bibitem[WY21]{WY19}
Julie Tzu-Yueh Wang and Yu~Yasufuku, \emph{Greatest common divisors of integral
  points of numerically equivalent divisors}, Algebra Number Theory \textbf{15}
  (2021), no.~1, 287--305. \MR{4226990}

\end{thebibliography}
\Addresses
\end{document}